\documentclass[11pt,dvipsnames,svgnames,table]{article}
\usepackage[T1]{fontenc}
\usepackage{graphicx,amsmath,amsthm,amsfonts,amssymb,microtype,thmtools,mathtools,anyfontsize,thm-restate,verbatim,tikz}
\usepackage{xcolor}
\usetikzlibrary{intersections}
\usetikzlibrary{external}
\usepackage{esvect}
\usepackage{wrapfig}
\usepackage{paralist, tabularx}
\usepackage{subcaption}
\usepackage{float}
\usepackage[shortlabels]{enumitem}
\setlist[itemize]{topsep=0ex,itemsep=0ex,parsep=0.4ex}
\setlist[enumerate]{topsep=0ex,itemsep=0ex,parsep=0.4ex}
\usepackage[unicode=true]{hyperref}
\hypersetup{
colorlinks=true,
breaklinks=true,
linkcolor={black},
citecolor={black},
urlcolor={blue!60!black},
pdftitle={On String Graphs with Bounded Maximum Degree},
pdfauthor={Nikolai Karol}}
\usepackage[noabbrev,capitalise]{cleveref}
\crefname{lem}{Lemma}{Lemmas}
\crefname{thm}{Theorem}{Theorems}
\crefname{cor}{Corollary}{Corollaries}
\crefname{prop}{Proposition}{Propositions}
\crefname{conj}{Conjecture}{Conjectures}
\crefname{open}{Open Problem}{Open Problems}
\crefname{question}{Question}{Questions}
\crefname{claim}{Claim}{Claims}
\crefformat{equation}{(#2#1#3)}
\Crefformat{equation}{Equation #2(#1)#3}

\newcommand{\header}[1]{\textbf{\textit{#1}}\textbf{:}}
\newcommand{\defn}[1]{\textcolor{Maroon}{\emph{#1}}}

\usepackage[longnamesfirst,numbers,sort&compress]
{natbib}
\makeatletter
\def\NAT@spacechar{~}
\makeatother
\usepackage[margin=31mm]{geometry}
\DeclarePairedDelimiter{\floor}{\lfloor}{\rfloor}

\DeclareMathOperator{\dist}{dist}

\DeclareMathOperator{\ltw}{ltw}
\DeclareMathOperator{\tw}{tw}
\DeclareMathOperator{\rtw}{rtw}

\DeclareMathOperator{\level}{level}
\DeclareMathOperator{\poly}{poly}
\renewcommand{\thefootnote}{\fnsymbol{footnote}}
\theoremstyle{plain}
\newtheorem{thm}[equation]{Theorem}
\newtheorem{lem}[equation]{Lemma}
\newtheorem{conj}[equation]{Conjecture}
\newtheorem{cor}[equation]{Corollary}
\newtheorem{prop}[equation]{Proposition}

\newtheorem{claim}{Claim}[thm]

\theoremstyle{definition}

\begin{document}

\author{
Nikolai Karol\footnotemark[2]}

\footnotetext[2]{School of Mathematics, Monash   University, Melbourne, Australia  (\texttt{nikolai.karol@monash.edu}).}

\sloppy

\title{\bf\boldmath String Graphs: Product Structure and Localised Representations}

\maketitle


\begin{abstract}

We investigate string graphs through the lens of graph product structure theory, which describes
complicated graphs as subgraphs of strong products of simpler building blocks. A graph $G$ is called a \defn{string graph} if its vertices can be represented by a collection $\mathcal{C}$ of continuous curves (called a \defn{string representation} of $G$) in a surface so that two vertices are adjacent in $G$ if and only if the corresponding curves in $\mathcal{C}$ cross. We prove that every string graph with bounded maximum degree in a fixed surface is isomorphic to a subgraph of the strong product of a graph with bounded treewidth and a path. This extends recent product structure theorems for string graphs. Applications of this result are presented. This product structure theorem ceases to be true if the `bounded maximum degree' assumption is relaxed to `bounded degeneracy'. For string graphs in the plane, we give an alternative proof of this result. Specifically, we show that every string graph in the plane has a `localised' string representation where the number of crossing points on the curve representing a vertex $u$ is bounded by a function of the degree of $u$.

Our proof of the product structure theorem also leads to a result about the treewidth of outerstring graphs, which qualitatively extends a result of Fox and Pach [\emph{Eur.~J.~Comb.} 2012] about outerstring graphs with bounded maximum degree. We extend our result to outerstring graphs defined in arbitrary surfaces.

\end{abstract}

\renewcommand{\thefootnote}
{\arabic{footnote}}

\section{Introduction} \label{section:intro}

For a finite collection $\mathcal{C}$ of sets, the \defn{intersection graph} $G$ of $\mathcal{C}$ has vertex set $\mathcal{C}$ where two sets in $\mathcal{C}$ are adjacent in $G$ if and
only if they have non-empty intersection. A \defn{string graph} is a graph isomorphic to the intersection graph of a collection of non-self-intersecting continuous curves (also called \defn{strings}) in a surface $\Sigma$ such that no three curves cross at a common point\footnote{Note that one may define a string graph as a graph isomorphic to the intersection graph of a collection of continuous curves in $\Sigma$. The additional properties of the curves in our definition can be ensured by a simple redrawing argument.}. Such a collection of curves is called a \defn{string representation} of this string graph in $\Sigma$. Most of the literature on this topic considers string graphs in the plane rather than on surfaces. This paper analyses both settings, with the main focus being on string graphs on arbitrary surfaces.

\citet{Benzer59} introduced the concept of string graphs in 1959, while studying patterns of mutations in DNA sequences. In 1966, \citet{Sinden66} independently considered the same concept in the study of electrical networks realisable by printed circuits. String graphs
were first formally defined by \citet{EET76} in 1976. Ever since, string graphs have been widely studied both from a theoretical and a practical point of view; see \citep{KGK86,MP-DM92,MP-DM93,Mat14,Pawlik14,SSS-JCSS03,Krat-JCTB91,Mat15,Lee17,JU17,SS-JCSS04,FP10,FP12,FP14,RW19,PRY20,Tom24,KM91,PT06,Kratochvil91,KSSS24,Davies25,CCTZ25} for example.

The primary goal of this paper is to explore structural properties of string graphs. There are several important structural results about string graphs. For example, \citet*{SSS-JCSS03} proved that recognising string graphs in the plane is in NP and conjectured that this result could be extended for string graphs on surfaces. Improving on the works of \citet{FP10,FP14} and \citet{Mat14}, \citet{Lee17} proved that every string graph with $m$ edges in a surface with Euler genus\footnote{See \cref{preliminaries} for omitted definitions.} $g$ has a balanced separator of size $\mathcal{O}_{g}(\sqrt{m})$. Very recently, \citet{Davies25} and \citet*{CCTZ25} independently proved that string graphs in the plane are quasi-isometric to planar graphs.

\subsection{Product Structure Theory} \label{subsectionproductstructure}

We analyse string graphs through the lens of graph product structure theory, which describes complicated graphs as subgraphs of strong products of simpler building blocks. Many key theorems in graph product structure theory can be expressed using the notion of row treewidth introduced by \citet{BDJMW22}. Say a graph $G$ is \defn{contained} in a graph~$H$ if $G$ is isomorphic to a subgraph of $H$. The \defn{row treewidth} of a graph $G$, denoted \defn{$\rtw(G)$}, is the minimum treewidth of a graph $H$ such that $G$ is contained in $H \boxtimes P$ for some path~$P$. Graph product structure theory originated in the prominent paper of \citet{DJMMUW20}, who
proved that planar graphs have bounded row treewidth. This result has been the key tool to resolve several major open problems regarding queue layouts~\citep{DJMMUW20}, centred colourings~\citep{DFMS21}, and universal graphs~\citep{EJM23,DEGJMM21}. Other graph classes have been shown to have bounded row treewidth, including graphs with bounded Euler genus \citep{DHHW22,DJMMUW20}, apex-minor-free graphs~\citep{DJMMUW20}, $k$-planar graphs~\citep{DMW23,HW24}, 
powers of bounded degree planar graphs~\citep{DMW23,HW24}, and $k$-matching-planar
graphs~\citep{HKW}.

We prove a product structure theorem for string graphs with bounded maximum degree on surfaces. This extends previous results of \citet*{DMW23} and \citet{HW24} about product structure of string graphs (see \cref{subsection:introlocal} for a detailed comparison).

\begin{thm} \label{thm:intromain} Let $G$ be a string graph with maximum degree at most $\Delta$ in a surface with Euler genus $g$. Then $\rtw(G) \leqslant f(\Delta, g)$ for some function $f$. That is, $G$ is contained in $H \boxtimes P$ for some graph $H$ of treewidth at most $f(\Delta, g)$ and for some path $P$.
\end{thm}

We now briefly describe some applications of \cref{thm:intromain}. Graph classes with bounded row treewidth have bounded queue number\footnote{The \defn{queue number} of a graph $G$ is the minimum integer $k$ such that there is a vertex ordering $\sigma$ of $V(G)$ and a partition $E_1,\dots,E_k$ of $E(G)$, such that for each $i\in\{1,\dots,k\}$, no two edges in $E_i$ are nested with respect to $\sigma$. Here edges $uw, xy \in E(G)$ with $\sigma(u) < \sigma(w)$ and $\sigma(x) < \sigma(y)$  are \defn{nested} with respect to $\sigma$ if $\sigma(u) < \sigma(x) < \sigma(y) < \sigma(w)$ or $\sigma(x) < \sigma(u) < \sigma(w) < \sigma(y)$.}~\citep{DJMMUW20}, polynomial $p$-centred chromatic number\footnote{The \defn{$p$-centred chromatic number} of a graph $G$ is the minimum number of colours in a vertex-colouring $\eta$ of $G$ such that for every connected subgraph $X$ of $G$, $|\{\eta(v) : v \in V(X)\}| > p$ or there exists some $v \in V(X)$ such that $\eta(v) \neq \eta(w)$ for every $w \in V(X) \setminus \{v\}$.}~\citep{DMW23}, bounded layered treewidth~\citep{BDJMW22}, and bounded strong and weak colouring numbers~\citep{vdHW18,KY03}. Classes of $n$-vertex graphs with bounded row treewidth also admit a universal\footnote{A graph $U$ is \defn{universal} for a graph class $\mathcal{G}$ if every graph of $\mathcal{G}$ is isomorphic to an induced subgraph of $U$.} graph with $n^{1+o(1)}$ vertices and edges~\citep{EJM23,DEGJMM21}. Thus, by \cref{thm:intromain}, string graphs with bounded maximum degree satisfy all of these applications.

Classes of $n$-vertex graphs with bounded row treewidth admit balanced separators of size~$\mathcal{O}(\sqrt{n})$~\citep{DMW17,BDJMW22}. Thus, in the `bounded maximum degree' setting, the structure in \cref{thm:intromain} is significantly stronger than the previous results of \citet{FP10,FP14}, \citet{Mat14}, and \citet{Lee17} mentioned above.

\subsection{Treewidth of Outerstring Graphs}

Our proof methods of \cref{thm:intromain} use the tool of coloured planarisations, recently introduced by \citet*{HKW} in the analysis of $k$-matching-planar drawings. This technique allows us to establish a result about the treewidth of outerstring graphs. Let $D$ be a closed disk in a surface. A collection of curves lying outside of $D$ with one endpoint on the boundary of $D$ is \defn{grounded} on $D$. A string graph $G$ is called \defn{outerstring} if $G$ has a string representation $\mathcal{C}$ grounded on some disk $D$ in the plane. A tuple $(\mathcal{C}, D)$ is called an \defn{outerstring diagram} of $G$.

Treewidth is of fundamental importance in
structural and algorithmic graph theory; see \citep{Bodlaender98, HW17, Reed03} for surveys. Treewidth is the standard measure of how similar a graph is to a tree. The treewidth of outerstring graphs is a popular research topic. For example, \citet{CDDFGHHWWY24} proved that the treewidth of every outerstring graph $G$ is upper-bounded by a function of the Hadwiger number of $G$. The next result we mention uses the concept of degeneracy. A graph $G$ is \defn{$d$-degenerate} if every subgraph of $G$ has minimum degree at most $d$. The \defn{degeneracy} of $G$ is the minimum integer $d$ such that $G$ is $d$-degenerate. \citet{SOX24} proved that $d$-degenerate $n$-vertex outerstring graphs have treewidth $\mathcal{O}(d\log n)$. \citet{FP-EJC12} proved that outerstring graphs with maximum degree at most $\Delta$ have balanced separators of size~$\mathcal{O}(\Delta)$. By a result of \citet{DN19}, this implies the following.

\begin{thm} [\citep{FP-EJC12}] \label{thm:FPouterstring} Every outerstring graph with maximum degree at most $\Delta$ has treewidth~$\mathcal{O}(\Delta)$.
    
\end{thm}

We qualitatively extend\footnote{Unlike the proof of \cref{thm:FPouterstring}, our proof of \cref{outerstringplane} does not require the result of \citet{DN19}.} \cref{thm:FPouterstring} using the following definitions. For an integer $t \geqslant 1$, an \defn{ordered $t$-colouring} of a graph $G$ is a function $\phi : V(G) \rightarrow \{1, \dots, t\}$ such that $\phi(v) \neq \phi(w)$ for each $vw \in E(G)$. Here the integers $1, \dots, t$ are called \defn{colours}. Intuitively, the colours are `ordered' from $1$ to~$t$. We say that $G$ is \defn{$(t, d)$-degenerate} if there exists an ordered $t$-colouring $\phi$ of $G$ such that for any vertex $v \in V(G)$, there are at most $d$ neighbours of $v$ with colour greater than $\phi(v)$. Note that every graph with maximum degree at most $\Delta$ is $(\Delta + 1, \Delta)$-degenerate, and every $(t,d)$-degenerate graph is $d$-degenerate.

\begin{thm} \label{outerstringplane} For each $t \geqslant 1$ and $d \geqslant 0$, every $(t, d)$-degenerate outerstring graph has treewidth at most $(3t - 1)(d + 1) - 1$.
\end{thm}

\cref{outerstringplane} cannot be strengthened using the standard concept of degeneracy. Indeed, \citet[Lemma~13]{SOX24} showed that the the class of $n$-vertex $2$-degenerate outerstring graphs has treewidth $\Omega(\log n)$. In light of this, \cref{outerstringplane} provides a setting sandwiched between `bounded maximum degree' and `bounded degeneracy', where the treewidth of outerstring graphs is bounded independently of the number of vertices.

We extend \cref{outerstringplane} using the following generalisation of outerstring graphs to arbitrary surfaces introduced by \citet{CDDFGHHWWY24}. Let $\Sigma$ be a surface with Euler genus $g$. A string graph $G$ in $\Sigma$ is called \defn{$(g, c)$-outerstring} if $G$ has a string representation $\mathcal{C}$ and there exist pairwise disjoint closed disks $D_{1}, \dots, D_{c}$ in $\Sigma$ where each curve $\gamma \in \mathcal{C}$ is drawn outside of $D_{1} \cup \dots \cup D_{c}$ and $\gamma$ is grounded on at least one of these disks. A tuple $(\Sigma, \mathcal{C}, \{D_{1}, \dots, D_{c}\})$ is called a \defn{$(g, c)$-outerstring diagram} of $G$. Note
that $(0, 1)$-outerstring graphs are precisely outerstring graphs.

\begin{thm} \label{outerstringsurfaces} For each $t, c \geqslant 1$ and $g, d \geqslant 0$, every $(t, d)$-degenerate $(g, c)$-outerstring graph has treewidth at most $(2t - 1)c(2g + 3)(d + 1) - 1$.
\end{thm}

Since every graph with maximum degree at most $\Delta$ is $(\Delta + 1, \Delta)$-degenerate, \cref{outerstringsurfaces} implies the following extension of \cref{thm:FPouterstring} to the setting of $(g, c)$-outerstring graphs.

\begin{cor} \label{intro:outerstringmaxdegree} For each $c \geqslant 1$ and $g \geqslant 0$, every $(g, c)$-outerstring graph with maximum degree at most $\Delta$ has treewidth at most $(2\Delta + 1)(\Delta + 1)c(2g + 3) - 1$.
\end{cor}

\subsection{Localised Representations} \label{subsection:introlocal}

This section describes an alternative proof of \cref{thm:intromain} for string graphs in the plane, and compares \cref{thm:intromain} with existing results about product structure of string graphs.

We first give a background on existing results about string representations of string graphs. For a string representation of a string graph, a \defn{crossing point} (or \defn{crossing}) is a common point of distinct curves. \citet{KM91} famously proved that there exist string graphs in the plane requiring exponentially many crossing points in any string representation. This result highlights the difficulty of analysing string graphs.

\begin{thm} [\citep{KM91}] \label{thm:KM91} For every integer $n \geqslant 2$, there exists a string graph $G$ in the plane on $n$ vertices such that in any string representation $\mathcal{C}$ of $G$ there exists a curve in $\mathcal{C}$ involved in at least $2^{cn}$ crossing points with curves in $\mathcal{C}$, where $c > 0$ is an absolute constant.
\end{thm}

Solving the long-standing open problem about the algorithmic decidability of string graphs \citep{Benzer59,Sinden66,G76,EET76}, \citet{SS-JCSS04}\footnote{Shortly before the work of \citet{SS-JCSS04}, \citet{PachToth-DCG02} claimed to prove a variant of \cref{Intro:SS}, in which the total number of crossing points in a string representation of a string graph $G$ is bounded by a function of $|V(G)|$. \citet*{KSSS24} discovered a serious gap in their proof which cannot be easily fixed.} proved that every string graph $G$ in the plane has a string representation 
in which the total number of crossing points
is bounded by a function of~$|E(G)|$. We state this result with a bound given by \citet[Theorem~9.17]{Schaefer18}, who slightly refined the argument of \citet{SS-JCSS04}.

\begin{thm} [\citep{Schaefer18,SS-JCSS04}] \label{Intro:SS} Every string graph in the plane with $m$ edges has a string representation with at most $2^m \cdot m^2$ crossing points.
\end{thm}

The main idea of the proof of \cref{Intro:SS} is to rephrase the problem in terms of drawings of graphs in the plane and apply a redrawing technique based on circular inversions. This approach makes essential use of properties of the plane and does not seem to generalise easily to other surfaces~\citep{SSS-JCSS03}. To tackle this issue, \citet*{SSS-JCSS03} rephrased the problem in terms of the solvability of word equations over trace monoids with involutions. They used this algebraic approach to prove the following qualitative generalisation of \cref{Intro:SS}.

\begin{thm} [\citep{SSS-JCSS03}] \label{intro:surfacesSSS} Every string graph in a compact surface $\Sigma$ with $n$ vertices has a string representation in $\Sigma$ with $\mathcal{O}_{\Sigma}(2^{2^{\poly(n)}})$ crossing points.    
\end{thm}

A string graph $G$ is a \defn{$(g, \delta)$-string} graph~\citep{DMW23,HW24} if there exists a string representation of $G$ in a surface with Euler genus $g$ where each curve is involved in at most $\delta$ crossing points. A \defn{$\delta$-string} graph is defined analogously for string graphs in the plane. So $\delta$-string graphs are exactly $(0, \delta)$-string graphs. Building on the work of \citet*{DMW23}, \citet{HW24} proved the following product structure theorem for $(g, \delta)$-string graphs.

\begin{thm} [\citep{HW24}] \label{gdStringGPST}
	Every $(g, \delta)$-string graph $G$ is contained in $ H\boxtimes P\boxtimes K_{2\max\{2g,3\}(\delta+1)^2}$ for some graph $H$ with treewidth at most $\binom{2 \floor{\delta/2}+4}{3}-1$ and for some path $P$, and thus $\rtw(G) \leqslant 2\max\{2g, 3\}(\delta+1)^2\binom{2 \floor{\delta/2}+4}{3} - 1$.
\end{thm}

\cref{thm:intromain} extends \cref{gdStringGPST}. Indeed, every $(g, \delta)$-string graph has maximum degree at most $\delta$. It is conceivable that string graphs in a surface with Euler genus $g$ with maximum degree at most $\Delta$ are not $(g, f(\Delta))$-string for any fixed function $f$ because two curves might cross arbitrarily many times. However, we show that this is not the case for string graphs in the plane. Specifically, we prove the following extension of \cref{Intro:SS}.

\begin{restatable}{thm}{intromainlocalised} \label{main:localised} For each string graph $G$ in the plane, there exists a string representation $\mathcal{C}$ of $G$ in the plane such that for each $u \in V(G)$, the curve in $\mathcal{C}$ representing $u$ is involved in at most $2^{d_{G}(u)}(d_{G}(u) - 1) + 1$ crossing points with other curves in $\mathcal{C}$.
\end{restatable}

For a string graph $G$ in the plane, \cref{Intro:SS} provides a bound on the number of crossing points on a curve of a string representation representing a vertex $u \in V(G)$ as a function of $|E(G)|$, which is a `global' parameter. The bound given by \cref{main:localised} is determined by a function of $d_{G}(u)$, which is a `local' parameter. In this sense, \cref{main:localised} shows that string graphs in the plane admit \textit{localised} string representations. Our proof of \cref{main:localised} is a refinement of the proof of \cref{Intro:SS} by \citet[Theorem~9.17]{Schaefer18}. \cref{main:localised} implies the following connection between string graphs in the plane with bounded maximum degree and $\delta$-string graphs.

\begin{cor} \label{intro:boundeddegree} Every string graph in the plane with maximum degree at most $\Delta$ is~${(2^{\Delta}(\Delta - 1) + 1)}$-string. 
\end{cor}

The bound $2^{\Delta}(\Delta - 1) + 1$ in \cref{intro:boundeddegree} cannot be improved to a subexponential function of~$\Delta$. Indeed, since every graph on $n$ vertices has maximum degree at most $n - 1$, \cref{thm:KM91} implies that there exists an absolute constant $c > 0$ such that for every positive integer $\Delta$ there exists a string graph with maximum degree $\Delta$ that is not $\delta$-string for any $\delta < 2^{c\Delta}$.

Applying \cref{intro:boundeddegree} with \cref{gdStringGPST}, we obtain the following product structure theorem for string graphs in the plane with bounded maximum degree. This provides an alternative proof of \cref{thm:intromain} for string graphs in the plane.

\begin{thm} \label{thm:PSmaxdegree}
	Every string graph $G$ in the plane with maximum degree at most $\Delta$ is contained in $H\boxtimes P\boxtimes K_{6(\delta+1)^2}$ for some graph $H$ with treewidth at most $\binom{2 \floor{\delta/2}+4}{3}-1$ and for some path $P$, and thus $\rtw(G) \leqslant 6(\delta+1)^2\binom{2 \floor{\delta/2}+4}{3}-1$, where $\delta := 2^{\Delta}(\Delta - 1) + 1$.
\end{thm}

We conjecture that string graphs on surfaces also admit localised string representations.

\begin{conj} \label{intro:conj} There exists a function $f$ such that for every surface $\Sigma$ with Euler genus $g$ and every string graph $G$ in $\Sigma$, there exists a string representation $\mathcal{C}$ of $G$ in $\Sigma$ such that for each $u \in V(G)$, the curve in $\mathcal{C}$ representing $u$ is involved in at most $f(d_{G}(u), g)$ crossings.
\end{conj}

If \cref{intro:conj} holds, it would extend \cref{intro:surfacesSSS} and generalise \cref{main:localised} to string graphs on surfaces. It would also imply that every string graph with maximum degree at most $\Delta$ in a surface with Euler genus $g$ is a $(g, h(g, \Delta))$-string graph for some function $h$. Combined with \cref{gdStringGPST}, this would provide an alternative proof of \cref{thm:intromain}. A potential way to settle \cref{intro:conj} is to refine the above-mentioned algebraic proof of \cref{intro:surfacesSSS} given by \citet{SSS-JCSS03}. However, this appears challenging, and it seems that new ideas are needed. In contrast, the techniques used in our proof of \cref{thm:intromain} are combinatorial, do not require the power of \cref{intro:surfacesSSS}, and allow us to prove \cref{outerstringplane,outerstringsurfaces}.

The paper is organised as follows. \cref{preliminaries} contains necessary definitions and a proof of a basic lemma of \citet*{KGK86}, which we use to prove our main results. \cref{section:colouredplanarisations} defines the coloured planarisation of a string representation and analyses its properties. \cref{section:proofs} proves our main results, \cref{thm:intromain,outerstringplane,outerstringsurfaces}. \cref{Section:localrealisations} proves \cref{main:localised}. In \cref{Section:intersectiongraphs}, we conjecture that the row treewidth of every string graph $G$ in a fixed surface is upper-bounded by a polynomial function of the maximum degree of $G$. As evidence for this conjecture, we show that it holds for intersection graphs of convex sets in the plane, which constitute an important subclass of string graphs. \cref{Section:examples} discusses examples that show that our product structure theorems cease to be true if the ‘bounded maximum degree’ assumption is relaxed to ‘bounded degeneracy'. \cref{section:conclusion} concludes the paper with open problems about string graphs with bounded degeneracy.

\section{Preliminaries} \label{preliminaries}

We consider finite undirected graphs without loops and parallel edges. For a vertex $v$ of a graph $G$, we denote its degree in $G$ by \defn{$d_{G}(v)$} and its neighbourhood by~\defn{$N_{G}(v)$}. For any undefined graph-theoretic terminology, see \citep{Diestel5}.

A \defn{drawing} of a graph $G$ is an embedding of $G$ in a surface, where vertices are associated with distinct points and each edge $vw$ of $G$ is associated with a non-self-intersecting curve between $v$ and $w$, such that:

\begin{itemize}

\item no edge passes through any vertex different from its endpoints,

\item  each pair of edges cross at a finite number of points,

\item  no three edges internally cross at a common point.

\end{itemize}

For a drawing $D$ of a graph $G$, a \defn{crossing} (or \defn{crossing point}) of distinct edges $e$ and $f$ of $G$ is an internal common point of the curves associated with $e$ and $f$ in $D$.

The \defn{Euler genus} of a surface obtained from a sphere by adding $h$ handles and $c$ cross-caps is~$2h+c$. The \defn{Euler genus} of a graph $G$ is the minimum Euler genus of a surface in which $G$ embeds without crossings.

A \defn{walk} in a graph $G$ is a sequence
$(v_{1}, v_{2},\dots, v_{s})$ of vertices in $G$ such that $v_{i}v_{i + 1} \in E(G)$ for each $i \in \{1,\dots, s - 1\}$. A \defn{path} in a graph $G$ is a walk $(v_{1}, v_{2},\dots, v_{s})$ in $G$ such that $v_{i} \neq v_{j}$ for all distinct $i, j \in \{1,\dots,s\}$. Let $W = (v_{1}, v_{2}, \dots, v_{s})$ be a walk in a graph $G$. We say that $v_{1}$ and $v_{s}$ are the \defn{endpoints} of $W$. For any $i \in \{1,\dots,i - 1\}$, $v_{i}$ and $v_{i + 1}$ are \defn{consecutive} vertices in $W$.

For a graph $G$, a \defn{tree decomposition} is a pair $(T, B)$ such that:

\begin{itemize}
    \item $T$ is a tree and $B : V(T) \rightarrow 2^{V(G)}$ is a function,
    \item for every edge $vw \in E(G)$, there exists a node $t \in V(T)$ with $v, w \in B(t)$, and
    \item for every vertex $v \in V(G)$, the subgraph of $T$ induced by $\{t \in V(T) : v \in B(t)\}$ is a non-empty (connected) subtree of $T$.
\end{itemize}

The sets $B(t)$ where $t \in V(T)$ are called \defn{bags} of $(T, B)$. The \defn{width} of a tree decomposition $(T, B)$ is $\max\{|B(t)| : t \in V(T)\} - 1$. The \defn{treewidth} of $G$, denoted \defn{$\tw(G)$}, is the minimum width of a tree decomposition of $G$.

Let $G$ and $H$ be graphs. $G$ is a \defn{minor} of $H$ if a graph isomorphic to $G$ can be obtained from $H$
by vertex deletion, edge deletion, and edge contraction. A \defn{model} of $G$ in $H$ is a function $\mu : V(G) \rightarrow 2^{V(H)}$ such that:
\begin{itemize}
    \item for each $v \in V(G)$, $\mu(v)$ is non-empty and the subgraph of $H$ induced by $\mu(v)$ is connected;

    \item $\mu(v) \cap \mu(w) = \emptyset$ for all distinct $v, w \in V(G)$; and

    \item for every edge $vw \in E(G)$, $ab \in E(H)$ for some $a \in \mu(v)$ and $b \in \mu(w)$.

\end{itemize}

The sets $\mu(v)$ are called \defn{branch sets} of $\mu$. It is folklore that $G$ is a minor of $H$ if and only if there exists a model of $G$ in $H$. It is well-known that if $G$ is a minor of $H$ then $\tw(G) \leqslant \tw(H)$ (see \citep{Bodlaender98} for an implicit proof).

A \defn{partition} of a graph $G$ is a set $\mathcal{P}$ of non-empty sets (called \defn{parts}) of vertices in $G$ such that each vertex of $G$ is in exactly one element of $\mathcal{P}$. A \defn{layering} is a partition $(V_{0}, V_{1},\dots, V_{s})$ of $G$ such that for every edge $vw \in E(G)$, if $v \in V_{i}$ and $w \in V_{j}$, then $|i - j| \leqslant 1$. Each set $V_{i}$ is called a \defn{layer}. The \defn{layered width} of a tree decomposition $(T, B)$ of $G$ is the minimum integer $\ell$ such that, for some layering $(V_{0}, V_{1},\dots, V_{s})$ of $G$, each bag $B(t)$ contains at most $\ell$ vertices in each layer $V_{i}$. The \defn{layered treewidth} of $G$, denoted \defn{$\ltw(G)$}, is the minimum layered width of a tree decomposition of $G$.

The \defn{strong product} of graphs $G_{1}$ and $G_{2}$ is the graph $G_{1} \boxtimes G_{2}$ with vertex set $V(G_{1} \boxtimes G_{2}) := \{(a, v) : a \in V(G_{1}), v \in V(G_{2})\}$, where distinct vertices $(a, v)$ and $(b, w)$ are adjacent if: $ab \in E(G_{1})$ and $v = w$; or $a = b$ and $vw \in E(G_{2})$; or $ab \in E(G_{1})$ and $vw \in E(G_{2})$. If $X \subseteq V(G_{1} \boxtimes G_{2})$, then the \defn{projection} of $X$ into $G_{1}$ is the set of vertices $a \in V(G_{1})$ such that $(a, v) \in X$ for some $v \in V(G_{2})$.

We include the proof of the result of \citet*{KGK86}, which states that every string graph in the plane has a string representation with a finite number of crossing points. In fact, the proof works for string graphs on surfaces.

\begin{lem} \label{lem:finitesurfaces} Every string graph $G$ in a compact surface $\Sigma$ has a string representation in $\Sigma$ with a finite number of crossing points.
\end{lem}

\begin{proof} Let $\mathcal{C} := \{\gamma_{1}, \dots, \gamma_{n}\}$ be a string representation of $G$ in $\Sigma$, and let $\Gamma := \gamma_{1} \cup \dots \cup \gamma_{n}$. Note that the curves $\gamma_{1}, \dots, \gamma_{n}$ are compact and $\Gamma$ is compact. For a point $p$ of $\Gamma$, let $\mathcal{N}_{p}$ be the collection of all open neighbourhoods $O_{p}$ of $p$ such that $O_{p}$ intersects only those curves of $\Gamma$ that contain $p$. If a compact set intersects all open neighbourhoods of a point $p$, it must contain $p$. Therefore, $\mathcal{N}_{p} \neq \emptyset$ for every point $p$ of $\Gamma$ because the curves of $\mathcal{C}$ are compact.

Let $\mathcal{O} := \bigcup_{p \in \Gamma}\mathcal{N}_{p}$. Observe that $\mathcal{O}$ is an open cover of $\Gamma$. Since $\Gamma$ is compact, there exists a finite cover $\mathcal{O}' \subseteq \mathcal{O}$ of $\Gamma$. Hence for each $i \in \{1, \dots, n\}$, the curve $\gamma_{i}$ is contained in a finite collection $\mathcal{O}'_{i} \subseteq \mathcal{O}'$ of open sets such that each open set of $\mathcal{O}'_{i}$ does not intersect any curve $\gamma_{j}$ that does not cross $\gamma_{i}$. Therefore every curve $\gamma_{i}$ can be replaced with a curve $\alpha_{i}$ drawn in $\mathcal{O}'_{i}$ such that: (i) $\alpha_{i}$ crosses $\alpha_{j}$ if and only if $\gamma_{i}$ crosses $\gamma_{j}$, and (ii) every two curves of $\{\alpha_{1}, \dots, \alpha_{n}\}$ have a finite number of crossing points in every open set of $\mathcal{O}'$. Note that (i) is possible because for every two crossing curves $\gamma_{i}$ and $\gamma_{j}$ there exists a set $O \in \mathcal{O}'$ that contains a common crossing point of $\gamma_{i}$ and $\gamma_{j}$. The condition (ii) guarantees the number of crossing points of the curves $\alpha_{1}, \dots, \alpha_{n}$ is finite because $\mathcal{O}'$ is finite.
\end{proof}

\section{Coloured Planarisations} \label{section:colouredplanarisations}

This section defines an auxiliary graph that is a useful tool in the proofs of our main results, \cref{thm:intromain,outerstringplane,outerstringsurfaces}. Specifically, we define the coloured planarisation of a string representation and analyse its
properties. The concept of coloured planarisations was recently introduced by \citet{HKW} in the analysis of $k$-matching-planar drawings.

In what follows, $\mathcal{C}$ is a string representation of a string graph $G$ without isolated vertices in a compact surface $\Sigma$. So every curve in $\mathcal{C}$ is involved in crossing points. By \cref{lem:finitesurfaces}, we can assume that $\mathcal{C}$ has a finite number of crossing points. By slightly perturbing the curves of $\mathcal{C}$, we can also ensure that: (i) every curve $\gamma$ in $\mathcal{C}$ has distinct endpoints (so $\gamma$ is not closed) and (ii) every endpoint of every curve $\gamma$ in $\mathcal{C}$ does not belong to any other curve in $\mathcal{C} \setminus \{\gamma\}$. Let \defn{$\mathcal{E}_{\mathcal{C}}$} be the set of endpoints of curves in $\mathcal{C}$. 

We define ordered colourings of $\mathcal{C}$ analogously to the definition of ordered colourings of graphs given in \cref{section:intro}. That is, an \defn{ordered $t$-colouring} of $\mathcal{C}$ is a function $\phi : \mathcal{C} \rightarrow \{1, \dots, t\}$ (where the integers $1, \dots, t$ are called \defn{colours}) such that no two curves in $\mathcal{C}$ of the same colour cross. We fix an ordered $t$-colouring $\phi$ of $\mathcal{C}$. Since $G$ has no isolated vertices, $t \geqslant 2$.

The \defn{planarisation} of $\mathcal{C}$, denoted \defn{$\mathcal{C}'$}, is the graph equipped with the drawing in $\Sigma$ obtained from $\mathcal{C}$ by replacing each crossing with a `dummy' vertex of degree $4$ and adding points of $\mathcal{E}_{\mathcal{C}}$ as vertices. Since every curve in $\mathcal{C}$ is involved in crossing points, $\mathcal{E}_{\mathcal{C}} \subsetneq V(\mathcal{C}')$. Since $\mathcal{C}$ has a finite number of crossing points (by \cref{lem:finitesurfaces}), $\mathcal{C}'$ is a finite graph. 
For each curve $\gamma \in \mathcal{C}$, let \defn{$L_{\gamma}$} be the path in $\mathcal{C}'$ determined by $\gamma$. Define the \defn{level} of a dummy vertex $d \in \gamma_{1} \cap \gamma_{2} \setminus \mathcal{E}_{\mathcal{C}}$ to be $\level(d) := \min(\phi(\gamma_{1}), \phi(\gamma_{2}))$. For each $v \in \mathcal{E}_{\mathcal{C}}$, let $\level(v) := 0$. Let \defn{$\mathcal{C}^{\phi}$} be the graph equipped with the drawing in $\Sigma$ obtained from $\mathcal{C}'$ as follows: for each curve $\gamma \in \mathcal{C}$ and for any two consecutive (along $\gamma$) dummy vertices $d_{1}, d_{2} \in L_{\gamma} \setminus \mathcal{E}_{\mathcal{C}}$ such that $\level(d_{1}) = \level(d_{2}) = \phi(\gamma)$, contract the edge $d_{1}d_{2}$ in $\mathcal{C}'$. We say that $\mathcal{C}^{\phi}$ is the \defn{coloured planarisation} of $\mathcal{C}$. Note that $\mathcal{C}'$ and $\mathcal{C}^{\phi}$ do not have crossings in $\Sigma$. So the Euler genus of $\mathcal{C}'$ and $\mathcal{C}^{\phi}$ is at most the Euler genus of $\Sigma$. In particular, if $\Sigma$ is the plane, then $\mathcal{C}'$ and $\mathcal{C}^{\phi}$ are planar graphs.

See \cref{fragmentssections,colouredplanarisation} for an example of a coloured planarisation. In these figures, the colours of the edges of $\mathcal{C}'$ and $\mathcal{C}^{\phi}$ are kept for better visual understanding, but formally speaking we do not define edge-colourings of $\mathcal{C}'$ or~$\mathcal{C}^{\phi}$. The endpoints $\mathcal{E}_{\mathcal{C}}$ are grey, and the vertices of $V(\mathcal{C}') \setminus \mathcal{E}_{\mathcal{C}}$ and $V(\mathcal{C}^{\phi}) \setminus \mathcal{E}_{\mathcal{C}}$ are black.

Let \defn{$\psi$} $: V(\mathcal{C}') \rightarrow V(\mathcal{C}^{\phi})$ be the surjective function determined by the contraction operation in the construction of $\mathcal{C}^\phi$. We emphasise that $\mathcal{C}^{\phi}$ depends upon the ordering of the colours in the ordered $t$-colouring $\phi$. Since every curve in $\mathcal{C}$ is involved in crossing points, $L_{\gamma} \setminus \mathcal{E}_{\mathcal{C}} \neq \emptyset$ and $W_{\gamma} \setminus \mathcal{E}_{\mathcal{C}} \neq \emptyset$ for each $\gamma \in \mathcal{C}$. Note that no edge of $\mathcal{C}'$ incident to $\mathcal{E}_{\mathcal{C}}$ is contracted in the construction of $\mathcal{C}^{\phi}$. So $\mathcal{E}_{\mathcal{C}} \subsetneq V(\mathcal{C}^{\phi})$ and $\psi(v) = v$ for each $v \in \mathcal{E}_{\mathcal{C}}$.

Let $\gamma \in \mathcal{C}$ be an arbitrary curve. The crossing points of $\gamma$ and the curves of colour less than $\phi(\gamma)$ split $\gamma$ into subcurves, called the \defn{fragments} of $\gamma$ (see \cref{fragmentssectionsa}). For each $\gamma \in \mathcal{C}$, every fragment of $\gamma$ naturally induces a subpath of $L_{\gamma}$. Let $M$ be such a subpath. If $M$ consists of at least three vertices, then the subpath of $M$ obtained by deleting the endpoints of $M$ is called a \defn{section} of $L_{\gamma}$ (see \cref{fragmentssectionsb}). By definition, every section $S$ of $L_{\gamma}$ is non-empty and the level of every vertex of $S$ is $\phi(\gamma)$.

\begin{figure}[h]
    \centering
    \begin{subfigure}[t]{0.44\textwidth}
    \centering
        \scalebox{1.13}{\includegraphics{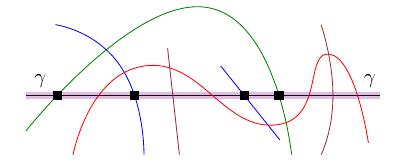}}
            \subcaption{$\mathcal{C}$}
            \label{fragmentssectionsa}
    \end{subfigure} \quad \quad  
    \begin{subfigure}[t]{0.5\textwidth}
    \centering
        \scalebox{1.13}{\includegraphics{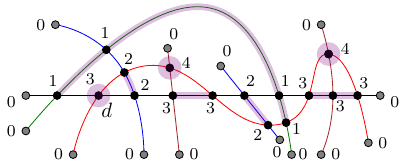}}
        \subcaption{$\mathcal{C}'$}
        \label{fragmentssectionsb}
    \end{subfigure}
    \caption{An example of fragments and sections. (a) A string representation $\mathcal{C}$ in the plane with an ordered $5$-colouring $\phi$, where colours are: green $= 1$, blue $= 2$, black $= 3$ (only the horizontal curve $\gamma$ is black), red $= 4$, and brown $= 5$. The  curve $\gamma$ is split by the crossing points (marked as squares) of $\gamma$ and the curves of smaller colours into five fragments (highlighted in purple). (b) The planarisation $\mathcal{C}'$ of $\mathcal{C}$. Each vertex is labelled by its level. The edges of sections and $1$-vertex sections of $\mathcal{C}'$ are highlighted in purple. There are three sections of $L_{\gamma}$, one of which consists of a single dummy vertex labelled $d$.}
    \label{fragmentssections}
\end{figure}

Let $S_{1}$ be a section of $L_{\gamma_{1}}$ and $S_{2}$ be a section of $L_{\gamma_{2}}$ such that $S_{1} \neq S_{2}$, where $\gamma_{1}, \gamma_{2} \in \mathcal{C}$. If $\gamma_{1} = \gamma_{2}$ then $S_{1}$ and $S_{2}$ are disjoint. Otherwise, $\gamma_{1} \neq \gamma_{2}$. If $S_{1} \cap S_{2} \neq \emptyset$ then there exists a dummy vertex $d \in S_{1} \cap S_{2} \setminus \mathcal{E}_{\mathcal{C}}$, and hence $d \in \gamma_{1} \cap \gamma_{2}$. By definition of $\phi$, we have $\phi(\gamma_{1}) \neq \phi(\gamma_{2})$. Without loss of generality, $\phi(\gamma_{1}) < \phi(\gamma_{2})$. By definition, $d$ is not a vertex of a section of $L_{\gamma_{2}}$, a contradiction. Thus, sections of $\mathcal{C}'$ are pairwise disjoint.

For each section $S$ of $\mathcal{C}'$, there exists exactly one curve $\gamma \in \mathcal{C}$ such that $L_{\gamma}$ contains $S$ and $\phi(\gamma)$ is equal to the common level of the vertices of $S$. The coloured planarisation $\mathcal{C}^{\phi}$ is obtained from the planarisation $\mathcal{C}'$ by contracting every edge in every section of $\mathcal{C}'$ (see \cref{colouredplanarisation}). For $v \in \mathcal{E}_{\mathcal{C}}$, we have $\psi^{-1}(v) = \{v\}$. Since sections are pairwise disjoint, for $x \in V(\mathcal{C}^{\phi}) \setminus \mathcal{E}_{\mathcal{C}}$, the set $\psi^{-1}(x)$ is the vertex set of a section of $\mathcal{C}'$. Note that $(\psi^{-1}(x) : x \in V(\mathcal{C}^{\phi}))$ is a partition of $\mathcal{C}'$.

For each vertex $x \in V(\mathcal{C}^{\phi})$, define the \defn{level} of $x$, denoted $\level(x)$, to be the common level of the vertices in $\psi^{-1}(x)$. Observe that the vertices of level $0$ are exactly the elements of $\mathcal{E}_{\mathcal{C}}$.

Consider any curve $\gamma \in \mathcal{C}$. Let $w_{0},\dots, w_{r}$ be the path $L_{\gamma}$ in $\mathcal{C}'$ (so $w_{0}$ and $w_{r}$ are the endpoints of $\gamma$). Let \defn{$W_{\gamma}$} be the walk in $\mathcal{C}^{\phi}$ obtained from $(\psi(w_{0}), \psi(w_{1}),\dots, \psi(w_{r}))$ by identifying consecutive identical vertices. Note that the level of each vertex in $W_{\gamma}$ is at most $\phi(\gamma)$. Observe that for distinct curves $\gamma_{1}, \gamma_{2} \in \mathcal{C}$, $W_{\gamma_{1}} \cap W_{\gamma_{2}} = \emptyset$ if and only if $\gamma_{1}$ and $\gamma_{2}$ cross.

\begin{figure}[h]
    \centering
        \scalebox{1.13}{\includegraphics{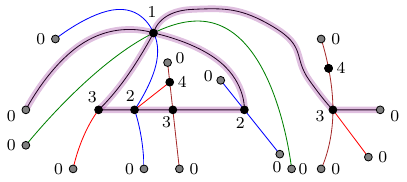}}
    \caption{The coloured planarisation $\mathcal{C}^{\phi}$ of the string representation $\mathcal{C}$ with the ordered $5$-colouring $\phi$ from \cref{fragmentssectionsa}. Each vertex is labelled by its level.
    The edges between consecutive vertices of the walk $W_{\gamma}$ in $\mathcal{C}^{\phi}$ are highlighted in purple.  
    }
    \label{colouredplanarisation}
\end{figure}

We now establish several properties of coloured planarisations.

\begin{lem} \label{lem:CPLuniquecurve}
    Let $x \in V(\mathcal{C}^{\phi}) \setminus \mathcal{E}_{\mathcal{C}}$. Then there exists exactly one curve $\gamma \in \mathcal{C}$ such that $\phi(\gamma) = \level(x)$ and $x \in W_{\gamma}$. Moreover, $L_{\gamma}$ contains $\psi^{-1}(x)$.
\end{lem}

\begin{proof}
    By assumption, $\psi^{-1}(x)$ is the vertex set of a section of $\mathcal{C}'$. Hence, there exists exactly one curve $\gamma \in \mathcal{C}$ such that $\phi(\gamma) = \level(x)$ and $L_{\gamma}$ contains $\psi^{-1}(x)$. Then $x \in W_{\gamma}$. Since no two curves in $\mathcal{C}$ of the same colour cross, no curve in $\mathcal{C}$ of colour $\phi(\gamma)$ crosses $\gamma$. Hence, $\gamma$ is the only curve in $\mathcal{C}$ of colour $\level(x)$ that contains a dummy vertex of $\psi^{-1}(x)$. Thus there is no curve $\gamma_{1} \in \mathcal{C} \setminus \{\gamma\}$ such that $\phi(\gamma_{1}) = \level(x)$ and $x \in W_{\gamma_{1}}$.
\end{proof}

\begin{lem} \label{lem:CPLlocaldistance} Let $Y \subseteq \mathcal{E}_{\mathcal{C}}$ be a set such that each curve in $\mathcal{C}$ has at least one of its endpoints in $Y$. Then for each $x \in V(\mathcal{C}^{\phi}) \setminus \mathcal{E}_{\mathcal{C}}$, there exists $v \in Y$ such that $\dist_{\mathcal{C}^{\phi}}(x, v) \leqslant t - 1$.
\end{lem}

\begin{proof}

Let $y \in V(\mathcal{C}^{\phi}) \setminus \mathcal{E}_{\mathcal{C}}$ be an arbitrary vertex. Recall that $\psi^{-1}(y)$ is the vertex set of a section of $\mathcal{C}'$. By \cref{lem:CPLuniquecurve}, there exists exactly one curve $\gamma \in \mathcal{C}$ such that $\phi(\gamma) = \level(y)$ and $L_{\gamma}$ contains $\psi^{-1}(y)$. Recall that the level of each vertex of $\psi^{-1}(y)$ is $\phi(\gamma)$. Since $L_{\gamma}$ is a path, there are two vertices $a, b$ in $L_{\gamma}$ adjacent in the graph $\mathcal{C}'$ to the endpoints of $\psi^{-1}(y)$. Observe that $\level(a) < \phi(\gamma)$ and $\level(b) < \phi(\gamma)$. Hence $y$ is adjacent to $\psi(a)$ and $\psi(b)$ in the graph $\mathcal{C}^{\phi}$, where $\level(\psi(a)) < \level(y)$ and $\level(\psi(b)) < \level(y)$. If $a, b \in \mathcal{E}_{\mathcal{C}}$ then $\psi(a) = a$ and $\psi(b) = b$, $a$ and $b$ are distinct, and at least one of $a$ or $b$ belongs to $Y$ because $a, b \in L_{\gamma}$. So if $y$ does not have neighbours of levels $1, 2, \dots, \level(y) - 1$, then $y$ is adjacent to a vertex of $Y$.

By definition, the level of each vertex in $\mathcal{C}'$ is at most $t - 1$. Therefore, the level of each vertex in $\mathcal{C}^{\phi}$ is at most $t - 1$. Hence $1 \leqslant \level(x) \leqslant t - 1$. By an observation in the previous paragraph, there exists a path $x = x_{0}, x_{1}, \dots, x_{r} = w$ in $G^{\phi}$ such that: (i) $\level(w) > 0$, (ii) $w$ does not have neighbours of levels $1, 2, \dots, \level(w) - 1$, and (iii) $\level(x_{i + 1}) < \level(x_{i})$ for each $i \in \{0,\dots,r - 1\}$. By an observation in the previous paragraph, there exists $v \in Y$ such that $vw \in E(\mathcal{C}^{\phi})$. It follows from (iii) that $\dist_{\mathcal{C}^{\phi}}(x, v) \leqslant t - 1$, as desired.
\end{proof}

\begin{lem} \label{lem:CPLconsecutive} For any $\gamma \in \mathcal{C}$, no two vertices with level $\phi(\gamma)$ are consecutive in $W_{\gamma}$.
\end{lem}

\begin{proof}

Assume for the sake of contradiction that some consecutive vertices $x, y$ in $W_{\gamma}$ have level $\phi(\gamma)$. By definition of $W_{\gamma}$, $x \neq y$. Since $\level(x) = \level(y) = \phi(\gamma)$, $\psi^{-1}(x)$ and $\psi^{-1}(y)$ are the vertex sets of some distinct sections $S_{1}$ and $S_{2}$ of $\gamma$. By definition of $W_{\gamma}$, there exist two dummy vertices $d_{x} \in S_{1}$, $d_{y} \in S_{2}$ such that $\psi(d_{x}) = x$, $\psi(d_{y}) = y$, and $d_{x}$, $d_{y}$ are consecutive vertices in the path $L_{\gamma}$. By definition of sections, no two dummy vertices of distinct sections of $L_{\gamma}$ are consecutive in $L_{\gamma}$, a contradiction.
\end{proof}

The next lemma is an ingredient in the proofs of our main results, \cref{thm:intromain,outerstringplane,outerstringsurfaces}. The proofs of these results consider models of string graphs in $H \boxtimes K_{s}$ for some $H$ and $s$. Throughout, we assume that $V(K_{s}) = \{1, \dots, s\}$.

\begin{lem} \label{lem:CPL} Let $G$ be a string graph without isolated vertices and $\mathcal{C} := \{\gamma_{v} : v \in V(G)\}$ be a string representation of $G$, where each curve $\gamma_{v}$ represents $v$. Let $\phi$ be an ordered $t$-colouring of $\mathcal{C}$ such that for any curve $\gamma_{v} \in \mathcal{C}$ and for any fragment $\alpha$ of $\gamma_{v}$, there are at most $d$ curves of colour greater than $\phi(\gamma_{v})$ that cross~$\alpha$. Then there exists a model $\mu$ of $G$ in $(\mathcal{C}^{\phi} - \mathcal{E}_{\mathcal{C}}) \boxtimes K_{d + 1}$ such that for each $v \in V(G)$, the projection of $\mu(v)$ into $\mathcal{C}^{\phi} - \mathcal{E}_{\mathcal{C}}$ is exactly $W_{\gamma_{v}} \setminus \mathcal{E}_{\mathcal{C}}$.
\end{lem}

\begin{proof}

For each $x \in V(\mathcal{C}^{\phi}) \setminus \mathcal{E}_{\mathcal{C}}$, let $B_{x} := \{v \in V(G) : x \in W_{\gamma_{v}}\}$. Recall that $\psi^{-1}(x)$ is the vertex set of some section $S$ of $L_{\gamma_{u}}$ for some $u \in V(G)$. By definition of $W_{\gamma_{v}}$, $x \in W_{\gamma_{v}}$ for some $v \in V(G)$ only if $L_{\gamma_{v}} \cap S \neq \emptyset$. Therefore, $x \in W_{\gamma_{v}}$ for some $v \in V(G)$ only if $v = u$ or $\gamma_{v}$ crosses the fragment $\alpha$ of $\gamma_{u}$ that corresponds to $S$. By definition of fragments, only the curves in $\mathcal{C}$ of colour greater than $i$ can cross $\alpha$. By assumption, the number of such curves is at most $d$. Thus $|B_{x}| \leqslant d + 1$ for each $x \in V(\mathcal{C}^{\phi}) \setminus \mathcal{E}_{\mathcal{C}}$.

For each $x \in V(\mathcal{C}^{\phi}) \setminus \mathcal{E}_{\mathcal{C}}$, let $\lambda_{x} : B_{x} \rightarrow \{1, \dots, d + 1\}$ be an injective function. For each $v \in V(G)$, define $\mu(v) := \{(x, \lambda_{x}(v)) : x \in W_{\gamma_{v}} \setminus \mathcal{E}_{\mathcal{C}}\}$. By construction, the projection of $\mu(v)$ into $\mathcal{C}^{\phi} - \mathcal{E}_{\mathcal{C}}$ is exactly $W_{\gamma_{v}} \setminus \mathcal{E}_{\mathcal{C}}$. Since $W_{\gamma_{v}} \setminus \mathcal{E}_{\mathcal{C}} \neq \emptyset$ for each $\gamma_{v} \in \mathcal{C}$, the subgraph of $(\mathcal{C}^{\phi} - \mathcal{E}_{\mathcal{C}}) \boxtimes K_{d + 1}$ induced by $\mu(v)$ is non-empty and connected. Since the functions $\lambda_{x}$ are injective, $\mu(v) \cap \mu(w) = \emptyset$ for all distinct $v, w \in V(G)$.

Let $vw \in E(G)$ be an edge. Then $\gamma_{v}$ and $\gamma_{w}$ cross. Hence there exists a vertex $y \in \mathcal{C}'$ such that $y \in L_{\gamma_{v}} \cap L_{\gamma_{w}}$. Therefore $\psi(y) \in W_{\gamma_{v}} \cap W_{\gamma_{w}}$. Hence $a:= (\psi(y), i) \in \mu(v)$ and $b:= (\psi(y), j) \in \mu(w)$ for some $i, j \in \{1, \dots, d + 1\}$. Therefore $ab \in E((\mathcal{C}^{\phi} - \mathcal{E}_{\mathcal{C}}) \boxtimes K_{d + 1})$ for some $a \in \mu(v)$ and $b \in \mu(w)$. Thus $\mu$ is a model of $G$ in $(\mathcal{C}^{\phi} - \mathcal{E}_{\mathcal{C}}) \boxtimes K_{d + 1}$.
\end{proof}

The next lemma bounds the distance between two vertices in the coloured planarisation and is used in the proof of \cref{thm:intromain}. The proof relies on the following definitions about walks. First, let $x$ and $y$ be vertices in a walk $W$ in a graph $H$ (possibly, $x = y$). Then we can enumerate the vertices of $W$ such that $W = (v_{1}, \dots, v_{s})$, $x = v_{i}$, $y = v_{j}$ for some $i \leqslant j$, and $x, y \notin (v_{i + 1}, \dots, v_{j - 1})$. Then $(x, v_{i + 1}, \dots, v_{j - 1}, y)$ is an \defn{$xy$-subwalk} of $W$. Note that there may exist several $xy$-subwalks of~$W$.

Second, let $W$ be a walk in a graph $H$ with distinct endpoints. Let $y$ be one of the endpoints of $W$. Then we can enumerate the vertices of $W$ such that $W = (v_{1}, \dots, v_{s})$ and $y = v_{1}$. Let $a$ be a vertex of $W$ such that $a \neq y$. Let $i \in \{2, \dots, t\}$ be the index such that $a = v_{i}$ and $a \neq v_{j}$ for any $j \in \{1,\dots,i - 1\}$. Then we say that $v_{i - 1}$ is the neighbour of $a$ \defn{towards} $y$ in $W$. Since $v_{1} \neq v_{s}$, the neighbour of $a$ towards $y$ in $W$ is unambiguously defined by $a$, $W$ and $y$. Let $b$ be a vertex of $W$. We say that $b$ is \defn{between} $a$ and $y$ in $W$ if there exists $j \in \{1, \dots, i\}$ such that $b = v_{j}$. In particular, $b$ is between $a$ and $y$ if $b \in \{a, y\}$. Observe that the neighbour of $a$ towards $y$ in $W$ is between $a$ and $y$ in $W$. If $b$ is between $a$ and $y$ in $W$ and a vertex $w$ is between $b$ and $y$ in $W$ then $w$ is between $a$ and $y$ in~$W$.

\begin{lem} \label{lem:CPLdistance} Let $\mathcal{C}$ be a string representation of a string graph $G$ without isolated vertices. Let $\phi$ be an ordered $t$-colouring of $\mathcal{C}$ such that for any curve $\gamma \in \mathcal{C}$, there are at most $k$ curves in $\mathcal{C}$ of colour smaller than $\phi(\gamma)$ that cross $\gamma$. Then for any $\gamma \in \mathcal{C}$ and any $x, y \in W_{\gamma} \setminus \mathcal{E}_{\mathcal{C}}$, we have $\dist_{\mathcal{C}^{\phi}}(x, y) \leqslant (2k + 1)\sum_{j = 0}^{t - 2}k^{j}$.
\end{lem}

\begin{proof}

Let $h(1) := 0$ and $h(i) := (2k + 1)\sum_{j = 0}^{i - 2}k^{j}$ for any $i \geqslant 2$. Observe that $h(i) = k(h(i - 1) + 2) + 1$. By induction, we prove that for any $\gamma \in \mathcal{C}$ and any $x, y \in W_{\gamma} \setminus \mathcal{E}_{\mathcal{C}}$, $\dist_{\mathcal{C}^{\phi}}(x, y) \leqslant h(\phi(\gamma))$. \cref{lem:CPLdistance} follows from this because $h(\phi(\gamma)) \leqslant h(t) = (2k + 1)\sum_{j = 0}^{t - 2}k^{j}$. We further assume that $x$ and $y$ are distinct, otherwise \cref{lem:CPLdistance} is trivial.

Consider the base case with $\phi(\gamma) = 1$. Note that the level of each vertex of $L_{\gamma} \setminus \mathcal{E}_{\mathcal{C}}$ is exactly $1$. Hence $L_{\gamma} \setminus \mathcal{E}_{\mathcal{C}}$ is a section of $L_{\gamma}$, and the edges of $L_{\gamma} \setminus \mathcal{E}_{\mathcal{C}}$ are contracted in the construction of $\mathcal{C}^{\phi}$. Therefore $|W_{\gamma} \setminus \mathcal{E}_{\mathcal{C}}| = 1$. Thus $x = y$, a contradiction.

Now assume that $\phi(\gamma) = i$ for some $i \in \{2, \dots, t\}$. Let $\Gamma$ be the set of curves of colour smaller than $i$ that cross $\gamma$. By assumption, $|\Gamma| \leqslant k$.

\begin{claim} \label{induction} For any $j \in \{0, \dots, |\Gamma|\}$, there exists a vertex $x_{j} \in W_{\gamma} \setminus \mathcal{E}_\mathcal{C}$ and a set $A_{j} \subseteq \Gamma$ such that:

\begin{itemize}
    \item $\dist_{\mathcal{C}^{\phi}}(x, x_{j}) \leqslant j(h(i - 1) + 2)$,
    \item $|A_{j}| = j$,
    \item there exists an $x_{j}y$-subwalk $W_{j}$ of $W_{\gamma}$ such that for each $\alpha \in A_{j}$, we have $W_{\alpha} \cap W_{j} \subseteq \{y\}$.
\end{itemize}
    
\end{claim}

\begin{proof}
    We prove this claim by induction on $j$. First, consider the base case $j = 0$. \cref{induction} is trivial for $j = 0$, $x_{0} := x$, $A_{0} := \emptyset$ and any $xy$-subwalk $W_{0}$ of $W_{\gamma}$.

    Now assume that $j \in \{1, \dots, |\Gamma|\}$. By the inductive hypothesis (for \cref{induction}), there exist a vertex $x_{j - 1} \in W_{\gamma} \setminus \mathcal{E}_\mathcal{C}$, a set $A_{j - 1} \subseteq \Gamma$ and an $x_{j - 1}y$-subwalk $W_{j - 1}$ of $W_{\gamma}$ such that all three properties in \cref{induction} are satisfied. Since $j \leqslant |\Gamma|$, we have $\Gamma \setminus A_{j - 1} \neq \emptyset$.

    If $x_{j - 1} = y$ then $\dist_{\mathcal{C}^{\phi}}(x, y) \leqslant (j - 1)(h(i - 1) + 2) \leqslant j(h(i - 1) + 2)$. In this case, let $x_{j} := y$, $A_{j} := A_{j - 1} \cup \{\alpha\}$ for any $\alpha \in \Gamma \setminus A_{j - 1}$, and $W_{j}$ be the walk consisting of a single vertex $y$. All three properties in \cref{induction} are satisfied for this choice of $x_{j}$, $A_{j}$, and $W_{j}$.

    Otherwise, $x_{j - 1} \neq y$. Let $u_{j}$ be the neighbour of $x_{j - 1}$ towards $y$ in $W_{j - 1}$.

    If $u_{j} = y$ then $\dist_{\mathcal{C}^{\phi}}(x, y) \leqslant 1 + (j - 1)(h(i - 1) + 2) \leqslant j(h(i - 1) + 2)$. In this case, let $x_{j} := y$, $A_{j} := A_{j - 1} \cup \{\alpha\}$ for any $\alpha \in \Gamma \setminus A_{j - 1}$, and $W_{j}$ be the walk consisting of a single vertex $y$. All three properties in \cref{induction} are satisfied for this choice of $x_{j}$, $A_{j}$, and $W_{j}$.

    Otherwise, $u_{j} \neq y$. Recall that the level of each vertex in $W_{\gamma}$ is at most $\phi(\gamma)$. By definition of $u_{j}$, the vertices $u_{j}$ and $x_{j - 1}$ are consecutive in $W_{j - 1}$ (and hence in $W_{\gamma}$). By \cref{lem:CPLconsecutive}, there exists a vertex $z_{j} \in \{x_{j - 1}, u_{j}\}$ such that $\level(z_{j}) < \phi(\gamma) = i$. So $\dist_{\mathcal{C}^{\phi}}(x_{j - 1}, z_{j}) \leqslant 1$, $z_{j}$ is between $x_{j - 1}$ and $y$ in $W_{j - 1}$, and $z_{j} \neq y$. By \cref{lem:CPLuniquecurve}, there exists a curve $\alpha_{j} \in \mathcal{C}$ such that $\phi(\alpha_{j}) = \level(z_{j}) < i = \phi(\gamma)$ and $z_{j} \in W_{\alpha_{j}} \cap W_{j - 1} \subseteq W_{\alpha_{j}} \cap W_{\gamma}$. So $\alpha_{j}$ crosses $\gamma$, and hence $\alpha_{j} \in \Gamma$. If $\alpha_{j} \in A_{j - 1}$, then by the inductive hypothesis, $W_{\alpha_{j}} \cap W_{j - 1} \subseteq \{y\}$, a contradiction to $z_{j} \in W_{\alpha_{j}} \cap W_{j - 1}$ and $z_{j} \neq y$. Thus $\alpha_{j} \notin A_{j - 1}$.

    Let $a_{j}$ be the first vertex of $W_{j - 1}$ starting at $y$ that is in $W_{\alpha_{j}}$. By the inductive hypothesis (for \cref{lem:CPLdistance}), $\dist_{\mathcal{C}^{\phi}}(a_{j}, z_{j}) \leqslant h(\phi(\alpha_{j})) \leqslant h(i - 1)$. Hence $\dist_{\mathcal{C}^{\phi}}(a_{j}, x) \leqslant \dist_{\mathcal{C}^{\phi}}(a_{j}, z_{j}) + \dist_{\mathcal{C}^{\phi}}(z_{j}, x_{j - 1}) + \dist_{\mathcal{C}^{\phi}}(x_{j - 1}, x) \leqslant h(i - 1) + 1 + (j - 1)(h(i - 1) + 2) \leqslant j(h(i - 1) + 2) - 1$. If $a_{j} = y$ then let $x_{j} := y$, $A_{j} := A_{j - 1} \cup \{\alpha_{j}\}$, and $W_{j}$ be the walk consisting of a single vertex~$y$. All three properties in \cref{induction} are satisfied for this choice of $x_{j}$, $A_{j}$, and $W_{j}$.

    Otherwise, $a_{j} \neq y$. It follows from definition of $a_{j}$ that $y \notin W_{\alpha_{j}}$. Let $x_{j}$ be the neighbour of $a_{j}$ towards $y$ in~$W_{j - 1}$. So $\dist_{\mathcal{C}^{\phi}}(x, x_{j}) \leqslant (j(h(i - 1) + 2) - 1) + 1 = j(h(i - 1) + 2)$. Let  $A_{j} := A_{j - 1} \cup \{\alpha_{j}\}$ and $W_{j}$ be the $x_{j}y$-subwalk of $W_{j - 1}$ such that $x_{j}$ occurs exactly one time in $W_{j}$ (that is, $W_{j}$ is the part of $W_{j - 1}$ between $y$ and the first occurence of $x_{j}$ in $W_{j}$ starting at $y$). All three properties in \cref{induction} are satisfied for this choice of $x_{j}$, $A_{j}$, and $W_{j}$.
\end{proof}

By \cref{induction} (setting $j = |\Gamma|$), there exists a vertex $x_{|\Gamma|} \in W_{\gamma} \setminus \mathcal{E}_{\mathcal{C}}$ such that $\dist_{\mathcal{C}^{\phi}}(x, x_{|\Gamma|}) \leqslant |\Gamma|(h(i - 1) + 2) \leqslant k(h(i - 1) + 2) = h(i) - 1$ and there exists an $x_{|\Gamma|}y$-subwalk $W_{|\Gamma|}$ of $W_{\gamma}$ such that for each $\alpha \in \Gamma$, we have $W_{\alpha} \cap W_{|\Gamma|} \subseteq \{y\}$. If $x_{|\Gamma|} = y$ then we are done. Otherwise, let $r_{0}$ be the neighbour of $x_{|\Gamma|}$ towards $y$ in $W_{|\Gamma|}$. Note that $\dist_{\mathcal{C}^{\phi}}(x, r_0) \leqslant h(i)$. If $r_{0} = y$ then we are done. Now assume that $r_{0} \neq y$. Note that $r_{0}$ and $x_{|\Gamma|}$ are consecutive in $W_{\gamma_v}$. By \cref{lem:CPLconsecutive}, there exists $z \in \{x_{|\Gamma|}, r_{0}\}$ such that $\level(z) < \phi(\gamma) = i$. By \cref{lem:CPLuniquecurve}, there exists a curve $\alpha \in \mathcal{C}$ such that $\phi(\alpha) = \level(z) < i = \phi(\gamma)$ and $z \in W_{\alpha_{j}} \cap W_{\gamma}$. So $\alpha$ crosses $\gamma$, and hence $\alpha \in \Gamma$. This contradicts $W_{\alpha} \cap W_{|\Gamma|} \subseteq \{y\}$.

We have shown that for any $\gamma \in \mathcal{C}$ and any $x, y \in W_{\gamma} \setminus \mathcal{E}_{\mathcal{C}}$, $\dist_{\mathcal{C}^{\phi}}(x, y) \leqslant h(\phi(\gamma))$. Since $h(\phi(\gamma)) \leqslant h(t) = (2k + 1)\sum_{j = 0}^{t - 2}k^{j}$, \cref{lem:CPLdistance} follows.
\end{proof}

\section{Proofs of Main Results}
\label{section:proofs}

This section proves our main results, Theorems \labelcref{thm:intromain}, \labelcref{outerstringplane} and \labelcref{outerstringsurfaces}. The proofs use coloured planarisations (\cref{section:colouredplanarisations}) and the well-known fact that $\tw(G \boxtimes K_{n}) \leqslant (\tw(G) + 1)n - 1$ for every graph $G$ and integer $n \geqslant 1$.

\subsection{Product Structure of String Graphs in Surfaces}

We start with \cref{thm:intromain}, which is our result about product structure of string graphs in surfaces. Our proof is based on the concept of weak shallow minors, recently introduced by \citet{HKW}. To explain this concept, several definitions are needed. Let $H$ be a graph and $A \subseteq V(H)$. The \defn{weak diameter} of $A$ in $H$ is the maximum distance in $H$ between the vertices of $A$; that is, $\max\{\dist_{H}(u,v) : u,v \in A\}$. \citet{HKW} used the following variant of this definition. The \defn{weak radius} of $A$ in $H$ is the minimum non-negative integer $r$ such that for some  $v \in V(H)$ and for every $a \in A$ we have $\dist_{H}(v, a) \leqslant r$. Note that the weak radius of $A$ in $H$ is less than or equal to the weak diameter of $A$ in $H$. A model $\mu$ of a graph $G$ in a graph $H$ is \defn{weak $r$-shallow} if the weak radius of every branch set of $\mu$ in $H$ is at most $r$. If there exists a weak $r$-shallow model of $G$ in $H$, then $G$ is a \defn{weak $r$-shallow minor} of $H$. We use the following result of \citet[Theorem~6.6]{HKW}.

\begin{thm} [\citep{HKW}] \label{thm:RTWmain}
Let $r, g \geqslant 0$ and $c \geqslant 1$ be integers. Let $H$ be a graph of Euler genus $g$ and $G$ be a weak $r$-shallow minor of $H \boxtimes K_c$. Then 
$$\rtw(G)\leqslant (4r + 1)c((2(8r+1)c+3)(2g + 7)^{(6r+2)(2g + 5) - 4} - 1) - 1.$$
\end{thm}

\begin{lem} \label{lem:PS} Let $G$ be a string graph in a surface $\Sigma$ with Euler genus $g$ and $\mathcal{C} = \{\gamma_{v} : v \in V(G)\}$ be a string representation of $G$ in $\Sigma$, where each curve $\gamma_{v}$ represents $v$. Suppose that for some $t \geqslant 2$ there exists an ordered $t$-colouring $\phi$ of $G$ such that:

\begin{itemize}
    \item for any curve $\gamma_{v} \in \mathcal{C}$ and for any fragment $\alpha$ of $\gamma_{v}$, there are at most $d$ curves of colour greater than $\phi(\gamma_{v})$ that cross $\alpha$,
    \item for any curve $\gamma_{v} \in \mathcal{C}$, there are at most $k$ curves in $\mathcal{C}$ of colour smaller than $\phi(\gamma_{v})$ that cross $\gamma_{v}$.
\end{itemize}

Then for $r := (2k + 1)\sum_{j = 0}^{t - 2}k^{j}$, we have

\begin{align*}
\ltw(G) &\leqslant (4r + 1)(d + 1)(2g + 3) \text{ and} \\
\rtw(G) &\leqslant (4r + 1)(d + 1)((2(8r+1)(d + 1) + 3)(2g + 7)^{(6r+2)(2g + 5) - 4} - 1) - 1.
\end{align*}

\end{lem}

\begin{proof} The row treewidth of a graph is the maximum row treewidth of its connected components, so we may assume that $G$ does not have isolated vertices.

By \cref{lem:CPL}, there exists a model $\mu$ of $G$ in $(\mathcal{C}^{\phi} - \mathcal{E}_{\mathcal{C}}) \boxtimes K_{d + 1}$ such that for each $v \in V(G)$, the projection of $\mu(v)$ into $\mathcal{C}^{\phi} - \mathcal{E}_{\mathcal{C}}$ is exactly $W_{\gamma_v} \setminus \mathcal{E}_{\mathcal{C}}$. By \cref{lem:CPLdistance}, every branch set of $\mu$ has weak diameter (and hence weak radius) at most $r = (2k + 1)\sum_{j = 0}^{t - 2}k^{j}$. Therefore $\mu$ is a weak $r$-shallow model, and hence $G$ is a weak $r$-shallow minor of $(\mathcal{C}^{\phi} - \mathcal{E}_{\mathcal{C}}) \boxtimes K_{d + 1}$. Recall that the graph $\mathcal{C}^{\phi}$ has Euler genus at most $g$. \citet{DMW17} proved that every graph with Euler genus $g$ has layered treewidth at most $2g + 3$. So $\ltw(\mathcal{C}^{\phi}) \leqslant 2g + 3$. \citet[Lemma~6.1]{HKW} proved that for any graph $H$ and any weak $r$-shallow minor $J$ of $H$, $\ltw(J) \leqslant (4r + 1)\ltw(H)$. Therefore $\ltw(G) \leqslant (4r + 1)(d + 1)(2g + 3)$. The bound on $\rtw(G)$ follows from \cref{thm:RTWmain}.
\end{proof}

Note that \cref{lem:PS} works in the setting that is slightly more general than the `bounded maximum degree' assumption. The following result, which implies \cref{thm:intromain}, follows from \cref{lem:PS}.

\begin{thm} \label{thm:PSspecificbounds} Let $G$ be a string graph with maximum degree at most $\Delta \geqslant 1$ in a surface with Euler genus $g$. Then for $r := (2\Delta + 1)\sum_{j = 0}^{\Delta - 1}\Delta^{j}$, we have
\begin{align*}
\ltw(G) &\leqslant (4r + 1)(\Delta + 1)(2g + 3) \text{ and} \\
\rtw(G) &\leqslant (4r + 1)(\Delta + 1)((2(8r+1)(\Delta + 1) + 3)(2g + 7)^{(6r+2)(2g + 5) - 4} - 1) - 1.
\end{align*}
\end{thm}

\begin{proof} Let $\phi$ be an ordered $(\Delta + 1)$-colouring of a string representation of $G$, where we `order' the colours arbitrarily. The result follows from \cref{lem:PS}, setting $t := \Delta + 1$, $d := \Delta$, and $k := \Delta$.
\end{proof}

\subsection{Treewidth of Outerstring Graphs}

This section proves our results about the treewidth of outerstring graphs (\cref{outerstringplane}) and $(g, c)$-outerstring graphs (\cref{outerstringsurfaces}). 

To prove \cref{outerstringplane}, we use the following classical result of \citet{RS-III}.

\begin{thm} [\citep{RS-III}] \label{PlanarBoundedRadius} Every planar graph with radius $r$ has treewidth at most $3r + 1$.
\end{thm}

The following result, which implies \cref{outerstringplane}, works in a setting that is slightly more general than the `$(t, d)$-degenerate' assumption of \cref{outerstringplane}.

\begin{thm} \label{thm:planarouterstring} Let $G$ be an outerstring graph. Suppose that there exists an outerstring diagram $(\mathcal{C}, D)$ of $G$ and an ordered $t$-colouring $\phi$ of $\mathcal{C}$ such that for any curve $\gamma \in \mathcal{C}$ and for any fragment $\alpha$ of $\gamma$, there are at most $d$ curves of colour greater than $\phi(\gamma)$ that cross $\alpha$. Then $\tw(G) \leqslant (3t - 1)(d + 1) - 1$.
\end{thm}

\begin{proof} 

Recall that $\mathcal{E}_{\mathcal{C}}$ is the set of endpoints of curves in $\mathcal{C}$. Let $Y \subseteq \mathcal{E}_{\mathcal{C}}$ be the set of such endpoints that lie on the boundary of $D$. By definition of outerstring graphs, each curve in $\mathcal{C}$ has at least one of its endpoints in $Y$.

Recall that the coloured planarisation $\mathcal{C}^{\phi}$ is obtained from the planarisation $\mathcal{C}'$ of $\mathcal{C}$ by contracting certain edges. We therefore may assume that the vertices and the edges of $\mathcal{C}$ lie outside of $D$ except for the vertices that belong to $Y$. Let $\mathcal{C}^{\phi}_{0}$ be the graph equipped with the drawing obtained from $\mathcal{C}^{\phi}$ by contracting the boundary of $D$ into a single point and deleting the vertices of $\mathcal{E}_{\mathcal{C}} \setminus Y$ and the edges incident to $\mathcal{E}_{\mathcal{C}} \setminus Y$. So in graph-theoretic sense, the vertices of $Y$ are identified to a single vertex $w$.

By \cref{lem:CPLlocaldistance}, for each $x \in V(\mathcal{C}^{\phi}) \setminus \mathcal{E}_{\mathcal{C}}$, there exists $v \in Y$ such that $\dist_{\mathcal{C}^{\phi}}(x, v) \leqslant t - 1$. Therefore for each $y \in V(\mathcal{C}^{\phi}_{0})$, we have $\dist_{\mathcal{C}^{\phi}_{0}}(y, w) \leqslant t - 1$. So $\mathcal{C}^{\phi}_{0}$ has radius at most $t - 1$. By \cref{PlanarBoundedRadius}, $\tw(\mathcal{C}^{\phi}_{0}) \leqslant 3t - 2$. By \cref{lem:CPL}, $G$ is a minor of $(\mathcal{C}^{\phi} - \mathcal{E}_{\mathcal{C}}) \boxtimes K_{d + 1}$. Hence $G$ is a minor of $\mathcal{C}^{\phi}_{0} \boxtimes K_{d + 1}$. Thus $\tw(G) \leqslant (3t - 1)(d + 1) - 1$.
\end{proof}

To prove our result about the treewidth of $(g, c)$-outerstring graphs, \cref{outerstringsurfaces}, we need the following extension of \cref{PlanarBoundedRadius}.

\begin{lem} \label{lem:radiuskverticesgenus} Let $G$ be a graph with Euler genus $g$. 
Suppose that there exist $c$ vertices $w_{1}, \dots, w_{c} \in V(G)$ such that for every $v \in V(G)$ there exists $i \in \{1, \dots, c\}$ such that $\dist_{G}(v, w_{i}) \leqslant r$. Then $\tw(G) \leqslant (2r + 1)c(2g + 3) - 1$.
\end{lem}

To prove \cref{lem:radiuskverticesgenus}, we use the concept of layered treewidth (see \cref{preliminaries} for a definition). \citet{DMW17} proved that every graph with Euler genus $g$ has layered treewidth at most $2g + 3$. Thus \cref{lem:radiuskverticesgenus} follows from \cref{lem:radiuskverticesltw}, which extends a result of \citet[Lemma~6]{DMW17}.

\begin{lem} \label{lem:radiuskverticesltw} Let $G$ be a graph. Suppose that there exist $c$ vertices $w_{1}, \dots, w_{c} \in V(G)$ such that for every $v \in V(G)$ there exists $i \in \{1, \dots, c\}$ such that $\dist_{G}(v, w_{i}) \leqslant r$. Then $\tw(G) \leqslant (2r + 1)c \cdot \ltw(G) - 1$.
\end{lem}

\begin{proof} Let $\ell := \ltw(G)$. So there exists a layering $(V_{0}, V_{1} \dots, V_{s})$ of $G$ and a tree decomposition $(T, B)$ of $G$ such that each bag contains at most $\ell$ vertices in each layer $V_{i}$. By assumption, $\mathcal{L}$ has at most $(2r + 1)c$ layers. Hence each bag of $(T, B)$ contains at most $(2r + 1)c\ell$ vertices. Therefore the width of $(T, B)$ is at most $(2r + 1)c\ell - 1$. Thus $\tw(G) \leqslant (2r + 1)c\ell - 1$.
\end{proof}

The following result, which implies \cref{outerstringsurfaces}, works in a setting that is slightly more general the `$(t, d)$-degenerate' assumption of \cref{outerstringsurfaces}.

\begin{thm} \label{thm:genusouterstring} Let $G$ be a $(g, c)$-outerstring graph. Suppose that there exists a $(g, c)$-outerstring diagram $(\Sigma, \mathcal{C}, \{D_{1}, \dots, D_{c}\})$ of $G$ and an ordered $t$-colouring $\phi$ of $\mathcal{C}$ such that for any curve $\gamma \in \mathcal{C}$ and for any fragment $\alpha$ of $\gamma$, there are at most $d$ curves of colour greater than $\gamma$ that cross $\alpha$. Then $\tw(G) \leqslant (2t - 1)c(2g + 3)(d + 1) - 1$.
\end{thm}

\begin{proof} Recall that $\mathcal{E}_{\mathcal{C}}$ is the set of endpoints of curves in $\mathcal{C}$. For each $i \in \{1, \dots, c\}$, let $Y_{i} \subseteq \mathcal{E}_{\mathcal{C}}$ be the set of such endpoints that lie on the boundary of $D_{i}$. Since the disks $D_{1}, \dots, D_{c}$ are disjoint, the sets $Y_{1}, \dots, Y_{c}$ are disjoint. Let $Y := Y_{1} \cup Y_{2} \cup \dots \cup Y_{c}$. By definition of $(g, c)$-outerstring graphs, each curve in $\mathcal{C}$ has at least one of its endpoints in $Y$.

Recall that the coloured planarisation $\mathcal{C}^{\phi}$ is obtained from the planarisation $\mathcal{C}'$ of $\mathcal{C}$ by contracting certain edges. We therefore may assume that the vertices and the edges of $\mathcal{C}$ lie outside of $D_{1} \cup D_{2} \dots \cup D_{c}$ except for the vertices that belong to $Y$. Let $\mathcal{C}^{\phi}_{0}$ be the graph equipped with the drawing obtained from $\mathcal{C}^{\phi}$ by contracting the boundary of each of $D_{i}$, where $i \in \{1, \dots, c\}$, into a single point and deleting the vertices of $\mathcal{E}_{\mathcal{C}} \setminus Y$ and the edges incident to $\mathcal{E}_{\mathcal{C}} \setminus Y$. So in graph-theoretic sense, for each $i \in \{1, \dots, c\}$, the vertices of $Y_{i}$ are identified to a single vertex $w_{i}$.

By \cref{lem:CPLlocaldistance}, for each $x \in V(\mathcal{C}^{\phi}) \setminus \mathcal{E}_{\mathcal{C}}$, there exists $v \in Y$ such that $\dist_{\mathcal{C}^{\phi}}(x, v) \leqslant t - 1$. Therefore for each $y \in V(\mathcal{C}^{\phi}_{0})$, there exists $i \in \{1, \dots, c\}$ such that $\dist_{\mathcal{C}^{\phi}_{0}}(y, w_{i}) \leqslant t - 1$. By \cref{lem:radiuskverticesgenus}, $\tw(\mathcal{C}^{\phi}_{0}) \leqslant (2t - 1)c(2g + 3) - 1$. By \cref{lem:CPL}, $G$ is a minor of $(\mathcal{C}^{\phi} - \mathcal{E}_{\mathcal{C}}) \boxtimes K_{d + 1}$. Hence $G$ is a minor of $\mathcal{C}^{\phi}_{0} \boxtimes K_{d + 1}$. Thus $\tw(G) \leqslant (2t - 1)c(2g + 3)(d + 1) - 1$.
\end{proof}

\section{String Graphs in the Plane} \label{Section:localrealisations}

This section proves \cref{main:localised}, which says that string graphs in the plane admit localised string representations.

We need the following definitions. Let $H$ be a graph and $R \subseteq \{\{e, f\} : e,f \in E(H)\}$ be a set of pairs of edges of $H$. A drawing $D$ of $H$ in the plane is a \defn{weak realisation}~\citep{Krat-JCTB91,KM91} of $(H, R)$ if every pair of crossing edges of $H$ belongs to $R$. In other words, only the pairs of edges that are in $R$ are allowed to cross in $D$. Note that the pairs of edges specified in $R$ do not \textit{have} to cross in a weak realisation of~$(H, R)$. We say that $(H, R)$ is \defn{weakly realisable} if it has a weak realisation. It is well-known that string graphs and weak realisability are closely related~\citep{Krat-JCTB91,KM91,SS-JCSS04}.

To prove \cref{Intro:SS}, \citet{SS-JCSS04} used the following result (see also the book of \citet[Lemma~9.2]{Schaefer18}). We use \cref{lem:weakrealisation} to prove \cref{main:localised}.

\begin{thm} [\citep{SS-JCSS04}]  \label{lem:weakrealisation} Let $H$ be a graph. Let $R \subseteq \{\{e, f\} : e,f \in E(H)\}$ be a set of pairs of edges of $H$ such that $(H, R)$ is weakly realisable. For each edge $e \in E(H)$, let $\delta_{e}$ be the number of pairs in $R$ that contain $e$. Then there exists a weak realisation $D$ of $(H, R)$ such that every edge $e \in E(H)$ is involved in at most $2^{\delta_{e}} - 1$ crossing points.
\end{thm}

We now prove \cref{main:localised}. Our proof is a refinement of the argument of \citet[Theorem~9.17]{Schaefer18} of the proof of \cref{Intro:SS}.

\intromainlocalised*

\begin{proof}

By \cref{lem:finitesurfaces}, there exists a string representation $\mathcal{C}_{0} = \{\gamma_{u} : u \in V(G)\}$ of $G$ with a finite number of crossing points, where for each $u \in V(G)$ the curve $\gamma_{u} \in \mathcal{C}_{0}$ represents $u$.  For each edge $uv \in E(G)$, fix a crossing point $x_{uv} \in \gamma_{u} \cap \gamma_{v}$, and let $X := \{x_{uv} : uv \in E(G)\}$ (see \cref{stringexample}). Hence for every $u \in V(G)$, there are exactly $d_{G}(u)$ points of $X$ on $\gamma_{u}$. Let $H$ be a graph and $D_{0}$ be a drawing of $H$ defined as follows. As illustrated in \cref{drawingexample}, let $V(H) := E(G)$ and let each vertex $uv \in E(G) = V(H)$ of $H$ be associated with $x_{uv}$ in the drawing $D_{0}$. Let $E(H)$ be the set of pairs $\{ua, ub\}$ such that $ua, ub \in E(G)$, $a \neq b$, and there are no points of $X$ on the subcurve of $\gamma_{u}$ between $x_{ua}$ and $x_{ub}$ (and $D_{0}$ associates the edge $\{ua, ub\}$ with this subcurve). In other words, for every $u \in V(G)$ such that $d_{G}(u) \geqslant 1$, the curve $\gamma_{u}$ is divided to $d_{G}(u) - 1$ edges of $H$ in the drawing $D_{0}$ by crossing points  $x_{uv_{1}}$, $x_{uv_{2}}$, $\dots$, $x_{uv_{d_{G}(u)}}$, where these points occur along $\gamma_{u}$ in this order. Let $\sigma_{u} := (v_{1}, \dots, v_{d_{G}(u)})$ be the order of the neighbours of $u$ corresponding to the order of these crossing points. Since the total number of crossing points of $\mathcal{C}_{0}$ is finite, $H$ is a finite graph and $D_{0}$ is indeed a drawing of~$H$ in the plane.

\begin{figure}[h]
    \begin{subfigure}[t]{0.5\textwidth}
    \centering
        \scalebox{1.3}{\includegraphics{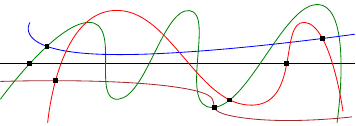}}
        \subcaption{$\mathcal{C}_{0}$ with the points of $X$}
        \label{stringexample}
    \end{subfigure}
    \begin{subfigure}[t]{0.55\textwidth}
    \centering
        \scalebox{1.3}{\includegraphics{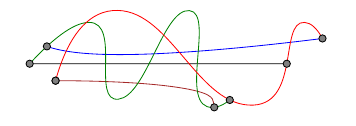}}
        \subcaption{The drawing $D_{0}$ of the graph $H$}
        \label{drawingexample}
    \end{subfigure}
    \caption{Proof of \cref{main:localised}. (a) A string representation $\mathcal{C}_{0}$. The points of $X$ are marked as black squares. For every two crossing curves, exactly one of their common crossing points is in $X$. (b) The drawing $D_{0}$ of the graph $H$ obtained from $\mathcal{C}_{0}$ and $X$. The vertices of $H$ are grey.}
    \label{example}
\end{figure}

Let $R$ be the set of pairs of edges of $H$ that cross in $D_{0}$. So $D_{0}$ is a weak realisation of~$(H, R)$. Since every curve $\gamma_{u} \in \mathcal{C}_{0}$ is non-self-intersecting, $\{\{ua, ub\}, \{uc, ud\}\} \notin R$ for every $ua, ub, uc, ud \in E(G) = V(H)$. By construction, every edge $\{ua, ub\} \in E(H)$ is contained in at most $d_{G}(u)$ pairs of $R$. By \cref{lem:weakrealisation}, there exists a weak realisation $D$ of $(H, R)$ such that every edge $\{ua, ub\}$ of $E(H)$ is involved in at most $2^{d_{G}(u)} - 1$ crossing points. For each $uv \in E(G) = V(H)$, let $y_{uv}$ be the point in the plane associated with~$uv$ in $D$. For each $u \in V(G)$, let $\alpha_{u}$ be the curve in the plane obtained by concatenation of $d_{G}(u) - 1$ curves $y_{uv_{1}}y_{uv_{2}}$, $y_{uv_{2}}y_{uv_{3}}$, $\dots$, $y_{uv_{d_{G}(u) - 1}}y_{uv_{d_{G}(u)}}$ of the drawing $D$, where $(v_{1}, \dots, v_{d_{G}(u)}) = \sigma_{u}$. Let $\mathcal{C} := \{\alpha_{u} : u \in V(G)\}$. Since $\{\{ua, ub\}, \{uc, ud\}\} \notin R$ for every $ua, ub, uc, ud \in E(G) = V(H)$, every curve $\alpha_{u}$ of $\mathcal{C}$ is non-self-intersecting. For each $uv \in E(G)$, the curves $\alpha_{u}$ and $\alpha_{v}$ have a common crossing point $y_{u, v}$. Since $D$ is a weak realisation of $(H, R)$, for each $u, v \in V(G)$ such that $uv \notin E(G)$, we have $\alpha_{u} \cap \alpha_{v} = \emptyset$. Therefore $\mathcal{C}$ is a string representation of $G$. For each $u \in V(G)$, there are two types of crossing points of $\alpha_{u}$ with other curves in $\mathcal{C}$: (i) the crossing points of $y_{uv_{i}}y_{uv_{i + 1}}$ with other edges in $D$, where $i \in \{1, \dots, d_{G}(u) - 1\}$ and (ii) $d_{G}(u)$ crossing points $y_{uv_{1}}, y_{uv_{2}}, \dots y_{uv_{d_{G}(u)}}$. By the choice of $D$, for each $i \in \{1, \dots, d_{G}(u) - 1\}$, there are at most $2^{d_{G}(u)} - 1$ crossing points of $y_{uv_{i}}y_{uv_{i + 1}}$ with other edges in $D$. Therefore the curve $\alpha_{u}$ is involved in at most $(2^{d_{G}(u)} - 1)(d_{G}(u) - 1) + d_{G}(u) = 2^{d_{G}(u)}(d_{G}(u) - 1) + 1$ crossing points, as desired.
\end{proof}

\section{Intersection Graphs of Convex Sets} \label{Section:intersectiongraphs}

Our bounds on the row treewidth of string graphs with bounded maximum degree are exponential for string graphs in the plane (\cref{thm:PSmaxdegree}) and even larger for string graphs on surfaces (\cref{thm:PSspecificbounds}). We conjecture that these bounds can be improved to a polynomial function.

\begin{conj} \label{conj:polynomial} Every string graph in a surface with Euler genus $g$ with maximum degree at most $\Delta$ has row treewidth at most $f(\Delta, g)$ for some polynomial function $f$.
\end{conj}

To provide evidence in support of \cref{conj:polynomial}, we show that it holds for intersection graphs of convex sets in the plane.

\begin{thm} \label{thm:intersectiongraphs} Every intersection graph $G$ of convex sets in the plane with maximum degree at most $\Delta$ is contained in $H \boxtimes P \boxtimes K_{6(2\Delta^{2} + 1)^{2}}$ for some graph $H$ with treewidth at most $\binom{2\Delta^{2} + 4}{3} - 1$ and for some path $P$, and thus $G$ has row treewidth at most $6(2\Delta^{2} + 1)^{2}\binom{2\Delta^{2} + 4}{3} - 1$.
\end{thm}

It is straightforward to show that every intersection graph of convex sets in the plane is a string graph in the plane. Remarkably, \citet*{PRY20} proved that almost all string graphs in the plane are intersection
graphs of convex sets in the plane. Thus \cref{thm:intersectiongraphs} provides substantial evidence for \cref{conj:polynomial}.

Little is known about product structure of intersection graphs of convex sets. \citet{MSSU24} analysed under which conditions intersection graphs of $\alpha$-free homothetic copies of a convex set have bounded row treewidth. See also the paper of \citet[Theorem~1.7]{HKW}, who proved the conjecture of \citet{MSSU24} about the product structure of $\alpha$-free homothetic copies of a regular convex polygon. Compared with these results, \cref{thm:intersectiongraphs} does not restrict shapes of convex sets but requires bounded maximum degree.

We now set out to prove \cref{thm:intersectiongraphs}. For an integer $k \geqslant 0$, a graph is \defn{$k$-planar} if it has a drawing in the plane such that every edge is involved in at most $k$ crossings. Building on the work of \citet{DMW23}, \citet{HW24} proved that every $k$-planar graph $G$ is contained in $H \boxtimes P \boxtimes K_{6(k + 1)^{2}}$ for some graph $H$ with treewidth at most $\binom{k + 4}{3} - 1$ and for some path $P$, and thus $G$ has row treewidth at most $6(k + 1)^{2}\binom{k + 4}{3} - 1$. Thus \cref{thm:intersectiongraphs} is implied by the following.

\begin{prop} \label{connection} Every intersection graph of convex sets in the plane with maximum degree at most $\Delta$ is $2\Delta^{2}$-planar.
\end{prop}

\begin{proof} Let $\mathcal{C} = \{S_{1}, \dots, S_{n}\}$ be a collection of convex sets in the plane such that the intersection graph $G$ of $\mathcal{C}$ has maximum degree at most $\Delta$. By slightly perturbing the sets of $\mathcal{C}$, if necessary, we can ensure that all sets of $\mathcal{C}$ are open. Recall that, by definition given in \cref{section:intro}, $V(G) = \{S_{1}, \dots, S_{n}\}$.

For each $i \in \{1, \dots, n\}$, take an arbitrary point $p_{i} \in S_{i}$. For each edge $S_{i}S_{j} \in E(G)$, take an arbitrary point $q_{i, j} \in S_{i} \cap S_{j}$. Since all sets of $\mathcal{C}$ are open, we can ensure that all taken points are pairwise distinct and
no three of them lie on a single straight line.

For each $i \in \{1, \dots, n\}$, associate $S_{i}$ with $p_{i}$. Associate each edge $S_{i}S_{j} \in E(G)$ with the $1$-bend polyline $\overline{p_{i}q_{i,j}p_{j}}$. Since all sets of $\mathcal{C}$ are open, we can also ensure that no three edges of $G$ internally cross at a common point in $D$. By construction, this association forms a drawing $D$ of $G$. For each $i \in \{1, \dots, n\}$, since $S_{i}$ is convex, the line segment $\overline{p_{i}q_{i,j}}$ is drawn completely inside $S_{i}$.

Consider any edge $S_{i}S_{j} \in E(G)$ and any line segment $\overline{p_{a}q_{a, b}}$ that crosses $\overline{p_{i}q_{i,j}}$. Since $\overline{p_{i}q_{i,j}}$ is drawn completely inside $S_{i}$ and $\overline{p_{a}q_{a,b}}$ is drawn completely inside $S_{a}$, we have $S_{i} \cap S_{a} \neq \emptyset$ and hence $S_{i}S_{a} \in E(G)$. Since the maximum degree of $G$ is at most $\Delta$, there are at most $\Delta$ choices for $a$, and for each such $a$ there are at most $\Delta$ choices for $b$. Hence there are at most $\Delta^{2}$ line segments $\overline{p_{a}q_{a, b}}$ crossing $\overline{p_{i}q_{i,j}}$. Therefore the part $\overline{p_{i}q_{i,j}}$ of the $1$-bend polyline $\overline{p_{i}q_{i,j}q_{j}}$ associated with the edge $S_{i}S_{j}$ of $G$ is involved in at most $\Delta^{2}$ crossing points. Similarly, the same holds for $\overline{q_{i,j}p_{j}}$. Hence there are at most $2\Delta^{2}$ crossing points of the $1$-bend polyline $\overline{p_{i}q_{i,j}q_{j}}$ associated with the edge $S_{i}S_{j}$ of $G$. Thus $D$ witnesses that $G$ is $2\Delta^{2}$-planar.
\end{proof}

Note that \cref{connection} is optimal up to a constant factor. For example, $K_{\Delta, \Delta}$ has maximum degree $\Delta$ and can be represented as the intersection graph of $\Delta$ vertical and $\Delta$ horizontal segments. This graph has $\Delta^{2}$ edges and crossing number $\Omega(\Delta^{4})$ (see \citep{BLNPSS23} for example). Thus some edge in every drawing of $K_{\Delta, \Delta}$ in the plane is involved in $\Omega(\Delta^{2})$ crossings.

\section{Bounded Maximum Degree is Necessary} \label{Section:examples}

This section discusses that the `bounded maximum degree' assumption in our product structure theorems cannot be relaxed to `bounded degeneracy'. Specifically, we discuss that three important subclasses of string graphs (intersection graphs of disks, grid intersection graphs, and outerstring graphs) with bounded degeneracy do not have bounded row treewidth and layered treewidth.

A lemma of \citet[Lemma~6]{DMW17} and a well-known relation between layered treewidth and row treewidth~\citep[Section~2]{BDJMW22} imply the following.

\begin{lem} [\citep{DMW17,BDJMW22}] \label{lem:knownrelations} For every graph $G$ with radius $r$, $$\rtw(G) + 1 \geqslant \ltw(G) \geqslant (\tw(G) + 1)/(2r + 1).$$
\end{lem}

\header{Intersection Graphs of Disks}  Grids constitute a classical example of graphs with large treewidth. The \defn{$(t \times t)$-grid} is the graph with vertex set $\{1, \dots, t\} \times \{1, \dots, t\}$ where vertices $(v_{1}, v_{2})$ and $(u_{1}, u_{2})$ are adjacent whenever $|v_{1} - u_{1}| + |v_{2} - u_{2}| = 1$ (see \cref{gridfigure}). The treewidth of the $(t \times t)$-grid is $t$ for every $t \geqslant 2$ (see \citep[Lemma~20]{HW17} for a proof). For every $t \geqslant 1$, there exists a $3$-degenerate intersection graph of disks isomorphic to the graph obtained from the $(t \times t)$-grid by adding a dominant vertex (see \cref{disksfigure}). This graph has radius~$1$. Thus by \cref{lem:knownrelations}, $n$-vertex intersection graphs of disks have layered treewidth $\Omega(\sqrt{n})$ and row treewidth $\Omega(\sqrt{n})$. This simple example was also mentioned by \citet{MSSU24}.

\begin{figure}[h]
    \begin{subfigure}[t]{0.5\textwidth}
    \centering
        \scalebox{0.9}{\includegraphics{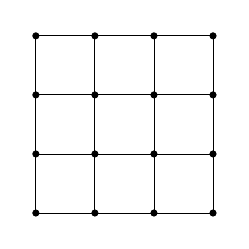}}
        \subcaption{$(4 \times 4)$-grid}
        \label{gridfigure}
    \end{subfigure}
    \begin{subfigure}[t]{0.5\textwidth}
    \centering
        \scalebox{0.9}{\includegraphics{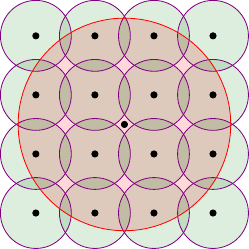}}
        \subcaption{Intersection graph of disks}
        \label{disksfigure}
    \end{subfigure}
    \caption{}
    \label{gridballs}
\end{figure}

\header{Grid Intersection Graphs} A \defn{grid intersection graph} is the intersection graph of a collection of horizontal and vertical segments in the plane such that no two parallel segments intersect. Grid intersection graphs form a natural subclass of string graphs and are well studied; see \citep{HNZ91,BHPW93,CFHW18} for example. Every grid intersection graph is also the intersection graph of axis parallel rectangles in the plane. Such graphs form a natural subclass of intersection graphs of convex sets and are well studied; see \citep{AG60,CW21,CFPS15} for example.

\begin{figure}[h]
        \centering
        \scalebox{1.07}{\includegraphics{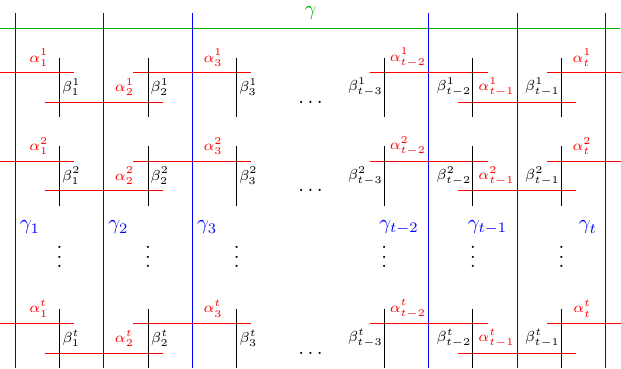}}
        \caption{A collection $\mathcal{C}$ of horizontal and vertical segments. The horizontal segment $\gamma$ (drawn in green) crosses $t$ vertical segments $\gamma_{1}, \dots, \gamma_{t}$ (drawn in blue) ordered `from left to right'. For each $i \in \{1, \dots, t\}$ and $j \in \{1, \dots, t\}$, the horizontal segment $\alpha_{i}^{j}$ (drawn in red) crosses $\gamma_{i}$ and no other segments of $\{\gamma_{1}, \dots, \gamma_{t}\}$. For each $i \in \{1, \dots, t - 1\}$ and $j \in \{1, \dots, t\}$, the vertical segment $\beta_{i}^{j}$ (drawn in black) crosses two segments $\alpha_{i}^{j}$, $\alpha_{i + 1}^{j}$ and crosses no other segment of $\mathcal{C}$.}
        \label{gridintersection}
\end{figure}

For a fixed integer $t \geqslant 1$, let $\mathcal{C}$ be the collection of horizontal and vertical segments defined and illustrated in \cref{gridintersection}. Observe that the intersection graph $G$ of $\mathcal{C}$ is a $2$-degenerate grid intersection graph. Note that $G$ is $K_{2, 2}$-free, has radius $3$ and $2t^{2} + 1$ vertices. 
    For each $j \in \{1, \dots, t\}$, let $X_{j} := \{\alpha_{1}^{j}, \beta_{1}^{j}, \alpha_{2}^{j}, \beta_{2}^{j}, \dots, \alpha_{t - 1}^{j}, \beta_{t - 1}^{j}, \alpha_{t}^{j}\}$. Note that the subgraph of $G$ induced by $X_{j}$ is a path. For each $i, j \in \{1, \dots, t\}$, the segment $\gamma_{i}$ is adjacent to the segment $\alpha_{i}^{j} \in X_{j}$.
    So $\gamma_{1}, \dots, \gamma_{t}$ and $X_{1}, \dots, X_{t}$ form a model of $K_{t, t}$ in $G$, and hence $K_{t, t}$ is a minor of $G$. Since $\tw(K_{t, t}) = t$, we have $\tw(G) \geqslant t$. Thus by \cref{lem:knownrelations}, $2$-degenerate $K_{2,2}$-free grid intersection graphs on $n$ vertices have layered treewidth $\Omega(\sqrt{n})$ and row treewidth $\Omega(\sqrt{n})$.

\header{Outerstring Graphs} \citet[Lemma~13]{SOX24} constructed $2$-degenerate $n$-vertex outerstring graphs with treewidth $\Omega(\log{n})$. Let $\mathcal{C}$ be a collection of $n$ curves in the plane grounded on a disk $D$ given by their construction. So the intersection graph $G$ of $\mathcal{C}$ is outerstring and $\tw(G) \in \Omega(\log{n})$. Since $\mathcal{C}$ is finite, there exists a curve $\gamma \notin \mathcal{C}$ outside of $D$ that has exactly one endpoint on the boundary of $D$ and crosses every curve in $\mathcal{C}$. Define $\mathcal{C}_{1} := \mathcal{C} \cup \{\gamma\}$. By construction, the intersection graph $H$ of $\mathcal{C}_{1}$ is outerstring and has radius $1$ and $n + 1$ vertices. Since $\tw(G) \in \Omega(\log{n})$ and $\tw(H) \geqslant \tw(G)$, we have $\tw(H) \in \Omega(\log{n})$. Thus by \cref{lem:knownrelations}, $3$-degenerate $n$-vertex outerstring graphs have layered treewidth $\Omega(\log{n})$ and row treewidth $\Omega(\log{n})$.

\section{Conclusion} \label{section:conclusion}

This paper shows that string graphs with bounded maximum degree have bounded row treewidth. \cref{Section:examples} discusses that the `bounded maximum degree' assumption cannot be replaced by `bounded degeneracy'. Still, it is desirable to have a strong structural description of $d$-degenerate string graphs in the spirit of product structure theory. It is not clear what such
a structural description should be. As a guide, one would expect that such a result would lead to positive solutions to the following open problems. Do $d$-degenerate string graphs have bounded queue number? Do $d$-degenerate string graphs have polynomial $p$-centred chromatic number? What is the size of the smallest universal graph for the class of $d$-degenerate $n$-vertex string graphs?

\subsection*{Acknowledgements}

Thanks to David Wood for helpful discussions and suggestions to improve this paper. Thanks to Kevin Hendrey and Jung Hon Yip for comments on an early draft of this paper.

{
\fontsize{10pt}{11pt}
\selectfont
\bibliographystyle{NikolaiNatbibStyle}
\bibliography{NikolaiBibliography}
}
\end{document}